\documentclass[a4paper,12pt,thmsa]{amsart}
\usepackage[a4paper,marginratio={1:1},scale={0.72,0.74},footskip=7mm,headsep=10mm]{geometry}

\usepackage{amsfonts}
\usepackage{amssymb,amsmath,latexsym}
\usepackage[dvips]{graphics}
\usepackage{graphicx,subfigure}
\usepackage[dvips]{color}
\usepackage[T1]{fontenc}
\usepackage[active]{srcltx}
\usepackage{amsmath}
\usepackage{amsfonts}
\usepackage{amssymb}
\usepackage{psfrag}
\usepackage{color}
\usepackage{url}
\usepackage{amsthm}
\usepackage{array}
\usepackage{pst-tree}
\usepackage{lscape}
\usepackage{dsfont}
\usepackage[T1]{fontenc}
\usepackage{pstricks,pstricks-add}\usepackage{mathrsfs}  

\newcommand{\llangle}{\langle\hspace*{-.03in}\langle}
\newcommand{\rrangle}{\rangle\hspace*{-.03in}\rangle}

\newtheorem{theorem}{Theorem}[section]
\newtheorem{proposition}{Proposition}[section]
\newtheorem{remark}{Remark}[section]
\newtheorem{definition}{Definition}[section]
\newtheorem{corollary}{Corollary}[section]

\newtheorem{lemma}{Lemma}[section]

\newcommand{\ed}{\stackrel{\mbox{\tiny $(law)$}}{=}}

\def\e{\mathbb{E}}
\def\p{\mathbb{P}}
\newcommand{\ind}{\mbox{\rm 1\hspace{-0.04in}I}}

\title[Fluctuation theory for spectrally positive additive L\'evy fields]
{Fluctuation theory for spectrally positive\\ additive L\'evy fields}

\author{Lo\"ic Chaumont}

\author{Marine Marolleau}

\address{Lo\"ic Chaumont -- LAREMA -- UMR CNRS 6093, Universit\'e d'Angers, 2 bd Lavoisier, 49045 Angers cedex~01}

\email{loic.chaumont@univ-angers.fr}

\address{Marine Marolleau -- LAREMA -- UMR CNRS 6093, Universit\'e d'Angers, 2 bd Lavoisier, 49045 Angers cedex~01}

\email{marine.marolleau@ens-rennes.fr}

\keywords{Additive L\'evy field, multivariate first hitting time, fluctuation theory, Kemperman's formula.}

\subjclass[2010]{60G51}

\thanks{}

\date{\today}

\begin{document}

\begin{abstract} A spectrally positive additive L\'evy field is a multidimensional field obtained as the sum 
$\mathbf{X}_{\rm t}={\rm X}^{(1)}_{t_1}+{\rm X}^{(2)}_{t_2}+\dots+{\rm X}^{(d)}_{t_d}$, ${\rm t}=(t_1,\dots,t_d)\in\mathbb{R}_+^d$, 
where ${\rm X}^{(j)}={}^t (X^{1,j},\dots,X^{d,j})$, $j=1,\dots,d$, are $d$ independent $\mathbb{R}^d$-valued L\'evy processes issued 
from 0, such that $X^{i,j}$ is non decreasing for $i\neq j$ and $X^{j,j}$ is spectrally positive. It can also be expressed as 
$\mathbf{X}_{\rm t}=\mathbb{X}_{\rm t}{\bf 1}$, where ${\bf 1}={}^t(1,1,\dots,1)$ and $\mathbb{X}_{\rm t}=(X^{i,j}_{t_j})_{1\leq i,j\leq d}$.
The main interest of spaLf's lies in the Lamperti representation of multitype continuous state branching processes. 
In this work, we study the law of the first passage times $\mathbf{T}_{\rm r}$ of such fields at levels $-{\rm r}$, where 
${\rm r}\in\mathbb{R}_+^d$. We prove that the field $\{(\mathbf{T}_{\rm r},\mathbb{X}_{\mathbf{T}_{\rm r}}),{\rm r}\in\mathbb{R}_+^d\}$ 
has stationary and independent increments and we describe its law in terms of this of the spaLf $\mathbf{X}$. In particular,
the Laplace exponent of $(\mathbf{T}_{\rm r},\mathbb{X}_{\mathbf{T}_{\rm r}})$ solves a functional equation leaded by the 
Laplace exponent of $\mathbf{X}$. This equation extends in higher dimension a classical fluctuation identity satisfied by the Laplace 
exponents of the ladder processes. Then we give an expression
of the distribution of $\{(\mathbf{T}_{\rm r},\mathbb{X}_{\mathbf{T}_{\rm r}}),{\rm r}\in\mathbb{R}_+^d\}$ in terms of the distribution of 
$\{\mathbb{X}_{\rm t},{\rm t}\in\mathbb{R}_+^d\}$ by the means of a Kemperman-type formula, well-known for spectrally positive 
L\'evy processes.
 \end{abstract}

\maketitle

%%%%%%%%%%%%%%%%%%%%%%%%%%%%%%%%%%%%%%%%%%%%%%%%%%%%%%%%%%%%%%%%%%%%%
\section{Introduction}
%%%%%%%%%%%%%%%%%%%%%%%%%%%%%%%%%%%%%%%%%%%%%%%%%%%%%%%%%%%%%%%%%%%%%

A spectrally positive, additive L\'evy field (spaLf) is defined by 
\[\mathbf{X}_{\rm t}:=\left(\sum_{j=1}^dX_{t_j}^{i,j},\,i=1,\dots, d\right)={\rm X}^{(1)}_{t_1}+\dots+{\rm X}^{(d)}_{t_d}\,,\;\;\;
{\rm t}=(t_1,\dots,t_d)\in\mathbb{R}_+^d\,,\]
where ${\rm X}^{(j)}={}^t (X^{1,j},\dots,X^{d,j})$, $j=1,\dots,d$, are $d$ independent $\mathbb{R}^d$-valued L\'evy processes such that 
$X^{i,j}$ are non decreasing for $i\neq j$ and $X^{j,j}$ is spectrally positive. SpaLf's can be considered as (non-trivial) extensions in 
higher dimension of spectrally positive L\'evy processes and the purpose of this article is to develop fluctuation theory for such random 
fields. The particular pathwise features of spaLf's allow us to define their first passage times 
$\mathbf{T}_{\rm r}=(T_{\rm r}^{(1)},\dots,T_{\rm r}^{(d)})$ at multivariate levels 
$-{\rm r}\in\mathbb{R}_-^d$ as the smallest of the indices ${\rm t}=(t_1,\dots,t_d)$ satisfying $\mathbf{X}_{\rm t}=-{\rm r}$ in the usual 
partial order of $\mathbb{R}^d$. The distribution of the variables $(\mathbf{T}_{\rm r},\mathbb{X}_{\mathbf{T}_{\rm r}})$, 
${\rm r}\in\mathbb{R}^{d}_{+}$ can then be related to the distribution of the field $\{\mathbb{X}_{\rm t},{\rm t}\in\mathbb{R}_+^d\}$, 
where $\mathbb{X}_{\rm t}=(X^{i,j}_{t_j})_{1\leq i,j\leq d}$. In doing so we obtain some fluctuation-type identities in the general 
framework of multivariate stochastic fields. These results provide an intrinsic motivation for the present study that can be considered 
in the line of several works on additive L\'evy processes from Khoshnevisan and Xiao, see for instance \cite{kx}.\\

The original motivation comes from an extension of the Lukasiewicz-Harris coding of Bienaym\'e-Galton-Watson trees through 
downward skip free random walks. In \cite{cl}, the authors proved that multitype Bienaym\'e-Galton-Watson trees can be coded by 
multivariate random fields 
\[\left(\sum\limits_{j=1}^d S_{n_j}^{i,j},\,i=1,\dots, d\right),\;\;\;\mbox{where}\;\;\;{\rm S}^{(j)}=(S^{1,j},\dots,S^{d,j}),\;\;\;j=1,\dots,d,\] 
are $d$ independent $\mathbb{Z}^d$-valued random walks such that $S^{i,j}$ are non decreasing for $i\neq j$ and $S^{j,j}$ is 
downward skip free. These random fields are the discrete time counterparts of spaLf's which suggests the possibility of coding 
continuous multitype branching trees in an analogous way. It seems quite complicated to achieve such a result as the notion of 
continuous multitype tree is not clearly defined for general mechanisms. 
However, reducing the analysis to processes rather than trees, one may still consider the Lamperti representation which provides 
a pathwise relationship between branching processes and their mechanism. This representation can be extended to 
continuous time multitype branching processes by using spaLf's. It was done in \cite{ch} for the discrete valued case and in \cite{cpu} 
and \cite{gt} for the continuous one. More specifically, let ${\rm Z}=(Z^{(1)},\dots,Z^{(d)})$ be a continuous time multitype 
branching process issued from ${\rm r}\in\mathbb{R}_+^d$. Then ${\rm Z}$ can be represented as the unique pathwise solution 
of the following equation,
\[(Z^{(1)}_t,\dots,Z^{(d)}_t)={\rm r}+\left(\sum_{j=1}^dX^{1,j}_{\int_0^t Z^{(j)}_s\,ds},
\dots,\sum_{j=1}^dX^{d,j}_{\int_0^t Z^{(j)}_s\,ds}\right),\;\;\;t\ge0\,,\]
where  ${\rm X}^{(j)}$, $j=1,\dots,d$, are L\'evy processes as described above.
Now recall that $0$ is an absorbing state for ${\rm Z}$. Then it follows from the above equation that the path
of ${\rm Z}$ up to its first passage time at 0 is entirely determined by the path of the spaLf 
\[\{\mathbf{X}_{\rm t}, {\rm t}\in\mathbb{R}^{d}_{+}\}=\left\lbrace\left(\sum\limits_{j=1}^dX_{t_j}^{i,j}\right)_{1\leq i,j\leq d},
{\rm t}\in\mathbb{R}^d_+\right\rbrace\] 
up to its first passage time $\mathbf{T}_{\rm r}$ at level $-{\rm r}$.
This fact  which is plain in the case $d=1$ will be proved in the general case in the upcoming paper \cite{cam2}, where extinction 
of continuous time multitype branching processes is characterized through path properties of spaLf's.\\ 

The next section consists in an important preliminary lemma for deterministic paths whose aim is to prove the existence of first 
passage times of spaLf's and to derive their first basic properties. Then in Section \ref{2018} we will turn our attention to the law of 
these first passage times. In particular we will prove that in analogy with the one dimensional case, their Laplace exponent is the 
inverse of the Laplace exponent of the spaLf. The situation for $d\ge2$ differs significantly from the one dimensional case as 
we first need to give necessary and sufficient conditions for the multivariate hitting times $\mathbf{T}_{\rm r}$ to be finite
on each coordinate, with positive probability, for all ${\rm r}\in\mathbb{R}_+^d$. (When $d=1$, this is equivalent to saying that 
the  spectrally positive L\'evy process is not a subordinator.) Another fundamental difference concerns the matrix valued field
$\mathbb{X}_{\mathbf{T}_{\rm r}}$ which is simply equal to ${\rm r}$ on the set $\mathbf{T}_{\rm r}<\infty$, when $d=1$. 
In Section  \ref{6932} we will focus on the law of the field $(\mathbf{T}_{\rm r},\mathbb{X}_{\mathbf{T}_{\rm r}})$ and prove 
that  its Laplace exponent solves a functional equation leaded by the Laplace exponent of the spaLf $\mathbf{X}$. This 
equation, see (\ref{2978}) in Theorem \ref{4126} below, can be compared to the classical Wiener-Hopf factorization involving the 
ladder processes of spectrally positive L\'evy processes. Then in Theorem \ref{3492}  the distribution of 
$(\mathbf{T}_{\rm r},\mathbb{X}_{\mathbf{T}_{\rm r}})$ will be fully characterized in terms of the distribution of the original stochastic
field $\mathbb{X}$, through an extension 
of Kemperman's formula, see Corollary VII.3 in \cite{be}.  More specifically, our result states that the measure 
\[\p(\mathbf{T}_{\rm r}\in {\rm dt},\,X^{i,j}_{t_j}\in {\rm d}x_{ij},\, 1\leq i,j\leq d)\,{\rm dr}\] 
is the image of the measure
\[\frac{\mbox{det}(-(x_{i,j})_{i,j\in[d]})}{t_1t_2\dots t_d}\prod_{j=1}^d\p(X^{i,j}_{t_j}\in {\rm d}x_{ij},\, i=1,\dots,d)
{\rm d}t_1\dots {\rm d}t_d,\]
through the mapping $({\rm t},(x_{i,j})_{i,j\in(d]})\mapsto({\rm t},(x_{i,j})_{i,j\in[d]},-(x_{i,j})_{i,j\in[d]}\cdot{\bf 1})$, 
where ${\bf 1}=(1,1,\dots,1)$. In order to prove it, we will use a similar identity recently obtained in \cite{cl} and \cite{ch} in the 
discrete time and space settings together with a discrete approximation. 

%%%%%%%%%%%%%%%%%%%%%%%%%%%%%%%%%%%%%%%%%%%%%%%%%%%%%%%%%%%%%%%%%%%%%
\section{A preliminary lemma in the deterministic setting}
%%%%%%%%%%%%%%%%%%%%%%%%%%%%%%%%%%%%%%%%%%%%%%%%%%%%%%%%%%%%%%%%%%%%%

We use the notation $\mathbb{R}_+=[0,\infty)$, $\overline{\mathbb{R}}_+=[0,\infty]$ and $[d]=\{1,\dots,d\}$, 
where $d\ge1$ is an integer. For ${\rm s}=(s_1,\dots,s_d)$ and ${\rm t}=(t_1,\dots,t_d)\in\overline{\mathbb{R}}_+^d$, 
we write ${\rm s}\le{\rm t}$ if $s_i\le t_i$ for all $i\in[d]$ and we write ${\rm s}<{\rm t}$ if ${\rm s}\le {\rm t}$ and there exists $i\in[d]$ 
such that $s_i<t_i$.\\

Recall that a real valued function $x:\mathbb{R}_+\rightarrow\mathbb{R}$ is said to be c\`adl\`ag, 
if it is right continuous on $\mathbb{R}_+$ and has left limits on $(0,\infty)$. Such a function is said to be downward 
skip free if for all $s\ge0$, $x(s)-x(s-)\ge0$, where we set $x(0-)=x(0)$. We also say that $x$ has no negative jumps. 
We will use the notation $x_t$ or $x(t)$ indifferently.

\begin{definition}\label{4141}
We call $\mathcal{E}_d$, the set of matrix valued functions $\textsc{x}=\{(x^{i,j}_{t_j})_{i,j\in[d]},\,{\rm t}\in\mathbb{R}^d_+\}$
such that for all $i,j$, $x^{i,j}$ is a c\`adl\`ag function and
\begin{itemize}
\item[$(i)$] $x^{i,j}_0=0$, for all $i,j\in[d]$,
\item[$(ii)$] for all $i\in[d]$, $x^{i,i}$ is downward skip free,
\item[$(iii)$]   for all $i,j\in[d]$ such that  $i\neq j$, $x^{i,j}$ is non decreasing.
\end{itemize}
\end{definition}

\noindent For ${\rm s}\in \overline{\mathbb{R}}_+^d$, we denote by $[d]_{\rm s}$ the set of indices of finite coordinates of 
${\rm s}$, that is $[d]_{\rm s}=\{i\in[d]:s_i<\infty\}$. For $i\neq j$, we set $x^{i,j}(\infty)=x^{i,j}(\infty-)=
\lim\limits_{s\rightarrow\infty} x^{i,j}(s)$.
\begin{definition}
Let $\textsc{x}\in\mathcal{E}_d$ and ${\rm r}=(r_1,\dots,r_d)\in\mathbb{R}_+^d$. Then ${\rm s}\in \overline{\mathbb{R}}_+^d$
is called a solution of the system $({\rm r},\textsc{x})$ if it satisfies
\begin{equation}\label{4379}
({\rm r},\textsc{x})\qquad r_i+\sum_{j=1}^dx^{i,j}(s_j-)=0\,,\;\;\;i\in [d]_{\rm s}\,.
\end{equation}
$($In particular, ${\rm s}=(\infty,\infty,\dots,\infty)$ is always a solution of the system $({\rm r},\textsc{x})$ since 
$[d]_{\rm s}=\emptyset$.$)$
\end{definition}

\noindent Note that in $(\ref{4379})$ it is implicit that $\sum\limits_{j\in [d]\setminus [d]_{\rm s}}x^{i,j}(s_j-)<\infty$, for all 
$i\in [d]_{\rm s}$, although by definition $s_j=\infty$, for $j\in [d]\setminus [d]_{\rm s}$. The next lemma is a continuous 
time and space counterpart of Lemma 1 in \cite{ch}. The proof of the present result follows a similar scheme, however we 
need to perform it here as it requires more care. 

\begin{lemma}\label{1377} Let $\textsc{x}\in\mathcal{E}_d$ and ${\rm r}=(r_1,\dots,r_d)\in\mathbb{R}_+^d$. 
\begin{itemize}
\item[$1.$] There exists a solution ${\rm s}=(s_1,\dots,s_d)\in\overline{\mathbb{R}}_+^d$ of the system $({\rm r},\textsc{x})$
such that any other solution ${\rm t}$ of $({\rm r},\textsc{x})$ satisfies ${\rm t}\ge {\rm s}$.  
The solution ${\rm s}$ will be called {\em the smallest solution} of the system $({\rm r},\textsc{x})$. 
\item[$2.$] Let ${\rm s}$ and ${\rm s}'$ be the smallest solutions of the systems $({\rm r},\textsc{x})$ and 
$({\rm r}',\textsc{x})$, respectively. If ${\rm r}'\le{\rm r}$, then ${\rm s}'\le{\rm s}$. Moreover if $({\rm r}_n)_{n\geq 0}$ is non
decreasing with $\lim\limits_{n\rightarrow\infty}{\rm r}_n={\rm r}$ then the sequence $({\rm s}_n)_{n\geq 0}$ of smallest 
solutions of $({\rm r}_n,\textsc{x})$ satisfies $\lim\limits_{n\rightarrow\infty}{\rm s}_n={\rm s}$.
\item[$3.$] Let ${\rm s}$ be the smallest solution of $({\rm r},\textsc{x})$. If ${\rm u}$ is such that  for all
$i\in[d]_{\rm u}$, $\sum\limits_{j=1}^dx^{i,j}(u_j-)$ $\le-r_i$, then ${\rm u}\ge{\rm s}$. As a consequence, for all 
${\rm u}\in\mathbb{R}^d_+$ such that ${\rm u}<{\rm s}$, there is $i\in[d]$ such that $\sum\limits_{j=1}^dx^{i,j}(u_j-)>-r_i$. 
\item[$4.$] The smallest solution ${\rm s}$ of $({\rm r},\textsc{x})$ satisfies $s_i=\inf\left\lbrace t:x^{i,i}_{t-}=
\inf\limits_{0\le u\le s_i}x^{i,i}_u\right\rbrace$, 
for all $i\in [d]_{\rm s}$.
\end{itemize}
\end{lemma}
\begin{proof} This proof is based on the observation that for each $i\in[d]$, as a function of ${\rm t}$, the 
term $\sum\limits_{j=1}^{d}x^{i,j}(t_j)$ has no negative jumps. Moreover, when $t_i$ is fixed, it is non decreasing.\\

Let us set $v_i^{(1)}=r_i$ and for $n\ge1$,
\[s_i^{(n)}=\inf\{t:x^{i,i}_{t-}=-v_i^{(n)}\}\;\;\mbox{and}\;\;v_{i}^{(n+1)}=r_i+\sum_{j\neq i}x^{i,j}(s^{(n)}_j-)\,,\]
where $\inf\emptyset=\infty$.  Set also ${\rm s}^{(0)}=0$ and note that $[d]_{{\rm s}^{(0)}}=[d]$. Then since for
$i\neq j$, the $x^{i,j}$'s are positive and non decreasing, we have 
\[{\rm s}^{(n)}\le {\rm s}^{(n+1)}\;\;\mbox{and}\;\;[d]_{{\rm s}^{(n+1)}}\subseteq [d]_{{\rm s}^{(n)}}\,,\;\;\;n\ge0\,.\]
Let us set ${\rm s}^{(\infty)}=\lim\limits_{n\rightarrow\infty}{\rm s}^{(n)}$.
Then ${\rm s}^{(\infty)}$ is the smallest solution of the system $({\rm r},\textsc{x})$ in the sense which is defined in Lemma 
\ref{1377}. Indeed, let $i\in[d]_{{\rm s}^{(\infty)}}$. By definition and since $x^{i,i}$ has no negative jumps, for all $n\ge1$, 
$x^{i,i}(s_i^{(n)}-)=-v_i^{(n)}$. Moreover, 
since the processes $t\mapsto x^{i,j}(t-)$ are left continuous, $\lim\limits_{n\rightarrow\infty}x^{i,i}(s_i^{(n)}-)=x^{i,i}(s_i^{(\infty)}-)$
and $\lim\limits_{n\rightarrow\infty}v_i^{(n)}=r_i+\sum\limits_{j\neq i}x^{i,j}(s_j^{(\infty)}-)$. Hence (\ref{4379}) is satisfied for 
${\rm s}^{(\infty)}$, that is $r_i+\sum\limits_{j=1}^dx^{i,j}(s_j^{(\infty)}-)=0$, for all $i\in[d]_{{\rm s}^{(\infty)}}$. Now let ${\rm t}\in\overline{\mathbb{R}}_+^d$ 
satisfying (\ref{4379}), that is
\begin{equation}\label{5389}
r_i+\sum_{j\neq i}x^{i,j}(t_j-)+x^{i,i}(t_i-)=0\,,\;i\in [d]_{{\rm t}}\,.
\end{equation}
We can prove by induction that ${\rm t}\ge {\rm s}^{(n)}$, for all $n\ge1$. Firstly for (\ref{5389}) to be satisfied, we should have 
$t_i\ge \inf\{s:x^{i,i}(s-)=-r_i\}$, for all $i\in [d]_{{\rm t}}$,
hence ${\rm t}\ge {\rm s}^{(1)}$. Now assume that ${\rm t}\ge {\rm s}^{(n)}$. Then $[d]_{{\rm t}}\subseteq [d]_{{\rm s}^{(n)}}$ and 
from (\ref{5389}), for each $i\in [d]_{{\rm t}}$, 
\[x^{i,i}(t_i-)=-\left(r_i+\sum_{j\neq i}x^{i,j}(t_j-)\right)\le-\left(r_i+\sum_{j\neq i}x^{i,j}(s_j^{(n)}-)\right)\,.\]
Therefore $t_i\ge \inf\left\lbrace s:x^{i,i}(s-)=-\left(r_i+\sum\limits_{j\neq i}x^{i,j}(s_j^{(n)}-)\right)\right\rbrace$, so that 
${\rm t}\ge {\rm s}^{(n+1)}$ and the first assertion is proved.\\

If ${\rm r}'\le {\rm r}$, then one can easily prove by induction that, with obvious notation, ${\rm s}'^{(n)}\le{\rm s}^{(n)}$ for all 
$n\ge1$ and the first part of assertion $2.$ follows. For the second part, set ${\rm s}':=\lim\limits_{n\rightarrow\infty}{\rm s}_n$. Then 
first part of assertion $2.$ yields ${\rm s}'\le{\rm s}$. Moreover, from the left continuity of the functions 
$t\mapsto x^{i,j}_{t-}$, $r_i+\sum\limits_{j=1}^dx^{i,j}(s_j'-)=0$, $i\in [d]_{{\rm s}'}$ hence ${\rm s}'$ is a solution of $({\rm r},\textsc{x})$ 
and thus ${\rm s}'={\rm s}$.\\

Let ${\rm u}\in\mathbb{R}^d_+$, such that $\sum\limits_{j=1}^dx^{i,j}(u_j-)\le-r_i$, for all $i\in[d]_{\rm u}$ and set $r'_i=
-\sum\limits_{j=1}^dx^{i,j}(u_j-)$. Since ${\rm r}'\ge{\rm r}$, it follows from 2.~that the smallest solution ${\rm s}'$ of the system 
$({\rm r}',\textsc{x})$ is such that ${\rm s}'\ge{\rm s}$. But since ${\rm u}$ is also a solution of $({\rm r}',\textsc{x})$, 1.~implies 
${\rm u}\ge{\rm s}'$ and the first assertion of 3.~follows. The second assertion of $3.$ is a consequence of the first one. Indeed, 
${\rm u}<{\rm s}$ implies that ${\rm u}\ge{\rm s}$ is not satisfied.\\

Assertion 4.~follows from the above construction of ${\rm s}={\rm s}^{(\infty)}$. Indeed, if there exists $i\in[d]_{\rm s}$ and
$t_{i}< s_{i}$ such that $x^{i,i}(t_{i}-)\le x^{i,i}(s_{i}-)$ then
\begin{equation}\label{9175}
\sum\limits_{j\neq i} x^{i,j}(s_{j}-)+x^{i,i}(t_{i}-) \le \sum\limits_{j\in[d]} x^{i,j}(s_{j}-)=-r_{i}
\end{equation}
and for all $k\in [d]_{\rm s}\setminus\{i\}$,
\begin{equation}\label{9275}
\sum\limits_{j\neq i} x^{k,j}(s_{j}-)+x^{k,i}(t_{i}-) \leq \sum\limits_{j\in[d]} x^{k,j}(s_{j}-)=-r_{k} \,.
\end{equation} 
Then set for all $k\in[d]_{s}$, $r'_k=-\left(\sum\limits_{j\neq i} x^{k,j}(s_{j}-)+x^{k,i}(t_i-)\right)$ and for all $k\in[d]\setminus[d]_{s}$, 
$r'_{k}=r_{k}$. Let ${\rm s}'$ be the smallest solution of 
the system $({\rm r}',\textsc{x})$. From part 2.~of the present lemma, since ${\rm r}'\ge{\rm r}$, ${\rm s}'\ge{\rm s}$. On the other 
hand, from (\ref{9175}),  (\ref{9275}) and part 3.~of the present lemma, 
${\rm s}>(s_1,\dots,s_{i-1},t_i,s_{i+1},\dots,s_d)\ge {\rm s}'$ which is a contradiction.
\end{proof}

\noindent We emphasize that according to our definition, some of the coordinates of the smallest 
solution of the system $({\rm r},\textsc{x})$ may be infinite. 

%%%%%%%%%%%%%%%%%%%%%%%%%%%%%%%%%%%%%%%%%%%%%%%%%%%%%%%%%%%%%%%%%%%%%
\section{Fluctuation theory for additive L\'evy fields}\label{2018}
%%%%%%%%%%%%%%%%%%%%%%%%%%%%%%%%%%%%%%%%%%%%%%%%%%%%%%%%%%%%%%%%%%%%%

Vectors of $\mathbb{R}^d$ will be denoted by ${\rm x}=(x_1,\dots,x_d)$ and ${\rm e}_i=(0,\dots,0,1,0,\dots,0)$ will be the $i$-th 
unit vector of $\mathbb{R}_+^d$. We will denote by $\langle {\rm x},{\rm y}\rangle$, ${\rm x},{\rm y}\in\mathbb{R}^d$ the usual 
scalar product on $\mathbb{R}^d$ and by $|{\rm x}|$ the euclidian norm of ${\rm x}$. 
A matrix $M=(m_{i,j})_{i,j\in[d]}\in M_d(\mathbb{R}\cup\{\infty\})$ is said to be irreducible if for all 
$i,j\in[d]$, there is a sequence $i=i_1,i_2,\dots,i_n=j$, for some $n\ge1$, such that $m_{i_k,i_{k+1}}\neq0$, for all $k=1,\dots,n-1$.
For two matrices $A$ and $B$ of $M_d(\mathbb{R})$, with columns ${\rm a}^{(1)},\dots,{\rm a}^{(d)}$ and 
${\rm b}^{(1)},\dots,{\rm b}^{(d)}$, respectively, we define the following special product,
\[\llangle A,B\rrangle=\sum\limits_{j\in[d]} \langle {\rm a}^{(j)},{\rm b}^{(j)}\rangle.\]
A matrix $A=(a_{i,j})_{i,j\in[d]}$ is called essentially nonnegative (or a Metzler matrix) if $a_{i,j}$ is 
nonnegative whenever $i\neq j$. For instance, for any element $\textsc{x}=\{(x^{i,j}_{t_j})_{i,j\in[d]},\,{\rm t}\in\mathbb{R}_+^d\}$ 
of the set $\mathcal{E}_d$ introduced at the previous section, the matrix $\textsc{x}_{\rm t}=(x^{i,j}_{t_j})_{i,j\in[d]}$ is essentially 
nonnegative for all $t_j\ge0$.\\

In this work, we shall consider $d$ independent L\'evy processes ${\rm X}^{(1)},\dots,{\rm X}^{(d)}$ on $\mathbb{R}_+^d$, such that 
with the notation ${\rm X}^{(j)}={}^t (X^{1,j},\dots,X^{d,j})$, for all $j\in[d]$, the process $X^{j,j}$ is a real spectrally positive L\'evy 
process, that is it has no negative 
jumps, and for all $i\neq j$, the L\'evy process $X^{i,j}$ is a subordinator. We emphasize that the processes $X^{1,j},\dots,X^{d,j}$ are 
not necessarily independent. Moreover, we do not exclude the possibility for a process $X^{i,j}$ to be identically equal to 0. 
It is known, see Chap.~VII, in \cite{be}, that the L\'evy process ${\rm X}^{(j)}$ admits all negative 
exponential moments. We denote by  $\varphi_j$ its Laplace exponent, that is
\[\e[e^{-\langle \lambda, {\rm X}^{(j)}_t\rangle}]=e^{t\varphi_j(\lambda)}\,,\;\;\;t\ge0\,,\;\;\;{\bf\lambda}=
(\lambda_1,\dots,\lambda_d)\in \mathbb{R}_+^d\,.\] 
Then from L\'evy Khintchine formula and the above assumptions on ${\rm X}^{(j)}$, $\varphi_j$ has the following form,
\begin{equation}\label{2552}
\varphi_j(\lambda)=-\sum_{i=1}^d a_{i,j}\lambda_i+\frac12q_j\lambda_j^2-
\int_{\mathbb{R}_+^d}(1-e^{-\langle\lambda, {\rm x}\rangle}-\langle\lambda, {\rm x}\rangle1_{\{|{\rm x}|<1\}})\,\pi_j(d{\rm x})\,, \;\; 
\lambda\in\mathbb{R}^{d}_{+},
\end{equation}
where $(a_{i,j})_{i,j\in[d]}$ is an essentially nonnegative matrix, $q_j\ge0$ and $\pi_j$ is a measure on $\mathbb{R}_+^d$ 
such that $\pi_j(\{0\})=0$ and
\[\int_{\mathbb{R}_+^d}\left[(1\wedge |{\rm x}|^2)+\sum\limits_{i\neq j}(1\wedge x_i)\right]\pi_j(d{\rm x})<\infty\,.\]
Note that for all $j\in[d]$, $\varphi_j$ is log-convex, i.e. the function $\log \varphi_j$ is convex on $(0,\infty)^d$. In particular, 
$\varphi_j$ is a convex function. Moreover, for all $i\neq j$ and $\lambda_1,\dots,\lambda_{i-1},\lambda_{i+1},\dots,\lambda_d$ 
the function $\lambda_i\mapsto \varphi_j(\lambda)$ is non increasing.\\

Let us now define the multivariate stochastic field 
\[\mathbf{X}_{\rm t}:={\rm X}_{t_1}^{(1)}+\dots+{\rm X}^{(d)}_{t_d}=\left(\sum_{j=1}^dX^{i,j}_{t_j}\right)_{i\in[d]}
\,,\;\;\;\mbox{for}\;\;\;{\rm t}=(t_1,\dots,t_d)\in\mathbb{R}_+^d\,.\]
Then $\mathbf{X}:=\{\mathbf{X}_{\rm t},\,{\rm t}\in\mathbb{R}_+^d\}$ 
is a particular case of additive L\'evy field in the sense of \cite{kx}. Its law is characterized by the Laplace exponent 
$\varphi:=(\varphi_1,\dots,\varphi_d)$, that is
\[\e[e^{-\langle\lambda, \mathbf{X}_{\rm t}\rangle}]=e^{\langle {\rm t},\varphi(\lambda)\rangle}\,,\;\;\;{\rm t},\lambda\in 
\mathbb{R}_+^d\,.\] 
Such an additive L\'evy field will be called a spectrally positive additive L\'evy field (spaLf). This terminology is justified by the 
results of this section which extend fluctuation theory for spectrally positive L\'evy processes. Let us also introduce the field of 
essentially nonnegative matrices 
\[\{\mathbb{X}_{\rm t}, {\rm t}\in\mathbb{R}^{d}_{+}\}=\{(X^{i,j}_{t_j})_{i,j\in[d]}, {\rm t}\in \mathbb{R}_+^d\}.\] 
Note that the spaLf $\mathbf{X}$ can be defined as $\mathbf{X}_{\rm t}=\mathbb{X}_{\rm t}\cdot{\bf 1}$,
where ${\bf 1}={}^t(1,1,\dots,1)$. Moreover, we emphasize that the spaLf $\mathbf{X}$ carries on the same 
information as the field of essentially nonnegative matrices $\{\mathbb{X}_{\rm t}, {\rm t}\in\mathbb{R}^{d}_{+}\}$.
For this reason, the terminology 'spaLf' will refer indifferently to $\mathbf{X}$ or to $\mathbb{X}$.
Let ${\rm r}=(r_1,\dots,r_d)\in\mathbb{R}_+^d$, since $\mathbb{X}\in\mathcal{E}_{d}$ a.s., according to Lemma \ref{1377} there is 
almost surely a smallest solution to the system 
\begin{equation}\label{1886}
({\rm r}, \mathbb{X}) \qquad \sum_{j=1}^dX^{i,j}_{s_j-}=-r_i,\,i\in[d]_{\rm s}\,.
\end{equation}
\noindent We will denote by $\mathbf{T}_{\rm r}=(T_{\rm r}^{(1)},\dots,T_{\rm r}^{(d)})$ this solution and use the notation
\begin{equation}\label{2678}
\mathbf{T}_{\rm r}=\inf\{{\rm t}:\mathbf{X}_{{\rm t}-}=-{\rm r}\},\;\;\mbox{with}\;\;\mathbf{X}_{{\rm t}-}=
\left(\sum_{j=1}^dX^{i,j}_{t_j-}\right)_{i\in[d]}\,.
\end{equation}
Then $\mathbf{T}_{\rm r}$ will be referred to as the (multivariate) first hitting time of level $-{\rm r}$ by the spaLf
$\{\mathbf{X}_{\rm t}, {\rm t}\in\mathbb{R}^{d}_{+}\}$. Note that according to Lemma \ref{1377}, some of the coordinates of 
$\mathbf{T}_{\rm r}$ can be infinite.

\begin{proposition}\label{8427}
Let $\mathbf{X}$ be a spaLf  and for ${\rm r}\in\mathbb{R}_+^d$, let $\mathbf{T}_{\rm r}$ be its 
first hitting time of level $-{\rm r}$ as defined above. Then,
\begin{itemize}
\item[$1.$] for all $j\in[d]$ and ${\rm r}\in\mathbb{R}_+^d$, ${\rm X}^{(j)}_{T_{\rm r}^{(j)}-}={\rm X}^{(j)}_{T_{\rm r}^{(j)}}$ 
a.s. on $\{T_{\rm r}^{(j)}<\infty\}$. In particular, for all $i\in[d]$,
\begin{equation}\label{3266}
\sum_{j=1}^dX^{i,j}_{T_{\rm r}^{(j)}-}=\sum_{j=1}^dX^{i,j}_{T_{\rm r}^{(j)}}=-r_i
\;\;\;\mbox{a.s. on the set $\{T_{\rm r}^{(i)}<\infty\}$}\,.
\end{equation}
\item[$2.$] For all ${\rm r}'\in\mathbb{R}_+^d$ such that $\p(\mathbf{T}_{{\rm r}'}\in\mathbb{R}_+^d)>0$, conditionally on 
$\{\mathbf{T}_{{\rm r}'}\in\mathbb{R}_+^d\}$, the field $\{\mathbf{T}_{{\rm r}+{\rm r}'}-\mathbf{T}_{{\rm r}'},\,{\rm r}\in\mathbb{R}_+^d\}$
has the same law as the field $\{\mathbf{T}_{\rm r},\,{\rm r}\in\mathbb{R}_+^d\}$ and it is independent of the field
$\{\mathbf{T}_{\rm r},\,{\rm r}\le{\rm r}'\}$.
In particular, for all ${\rm r},{\rm r}'\in\mathbb{R}_+^d$, 
\begin{equation}\label{3683}
\mathbf{T}_{{\rm r}+{\rm r}'}\ed \mathbf{T}_{\rm r}+ \tilde{\mathbf{T}}_{{\rm r}'},
\end{equation}
where $\tilde{\mathbf{T}}_{{\rm r}'}$ is an independent copy of $\mathbf{T}_{{\rm r}'}$. 
\item[$3.$] If $\p(\mathbf{T}_{\rm r}\in\mathbb{R}_+^d)>0$ for some ${\rm r}\in(0,\infty)^d$, then 
$\p(\mathbf{T}_{\rm r}\in\mathbb{R}_+^d)>0$ for all ${\rm r}\in\mathbb{R}^d_+$. Under this condition,
there is a mapping $\phi=(\phi_1,\dots,\phi_d):\mathbb{R}_+^d\rightarrow(0,\infty)^d$ such that 
\begin{equation}\label{5147}
\e[e^{-\langle\lambda,\mathbf{T}_{\rm r}\rangle}]=
e^{-\langle\phi(\lambda),{\rm r}\rangle},\;\;\; \lambda\in\mathbb{R}_+^d\,.
\end{equation}
Moreover, the mapping $\phi$ is differentiable and each $\phi_i$ is a concave function. 
\end{itemize}
\end{proposition}
\begin{proof}
The first assertion is a consequence of quasi-left continuity for L\'evy processes. Indeed, let us denote by 
$(\mathcal{F}_t^{(j)})_{t\geq 0}$ the natural filtration generated by ${\rm X}^{(j)}$ and set 
$\mathcal{F}_\infty^{(j)}=\sigma\left(\bigcup\limits_{t\ge0}\mathcal{F}_t^{(j)}\right)$. Then for all $t_j\ge0$, the set 
\[\{T_{\rm r}^{(j)}\le t_j\}=\bigcup\limits_{\stackrel{{\rm u}\in
(\mathbb{Q}\cup\{\infty\})^d}{u_j=t_j}}
\left\lbrace\exists \;{\rm s}\le{\rm u}:r_i+\sum_{k=1}^dX^{i,k}_{s_k-}=0,\;\;i\in[d]_{\rm s}\right\rbrace\]
belongs to the sigma-field $\mathcal{G}_{t_j}^{(j)}:=\sigma\left(\mathcal{F}_{t_j}^{(j)}
\cup\left(\bigcup\limits_{i\neq j} \mathcal{F}_\infty^{(i)}\right)\right)$, so that $T_{\rm r}^{(j)}$ is a stopping time of the filtration 
$(\mathcal{G}_{t}^{(j)})_{t\ge0}$. Moreover, since the processes ${\rm X}^{(i)}$, $i\in[d]$ are independent, ${\rm X}^{(j)}$ 
is a L\'evy process in the latter filtration. Now let us consider the sequence $(\mathbf{T}_{{\rm r}_{n}})_{n\geq 1}$, where
${\rm r}_{n}={\rm r}-{\rm e}_j/n$. Then from part 2.~of Lemma \ref{1377}, $T_{{\rm r}_n}^{(j)}$ is an increasing sequence of 
$(\mathcal{G}_{t}^{(j)})$-stopping times and this sequence satisfies 
$\lim\limits_{n\rightarrow\infty}T_{{\rm r}_n}^{(j)}=T_{{\rm r}}^{(j)}$. Therefore from quasi-left continuity of ${\rm X}^{(j)}$, 
see Proposition I.7 in \cite{be}, ${\rm X}^{(j)}_{T_{\rm r}^{(j)}-}={\rm X}^{(j)}_{T_{\rm r}^{(j)}}$ a.s. on 
$\{T_{\rm r}^{(j)}<\infty\}$. It clearly implies (\ref{3266}).\\

In order to prove $2.$ it suffices to see that conditionally on $\{\mathbf{T}_{{\rm r}'}\in\mathbb{R}^{d}_{+}\}$, 
the stochastic field $\{\tilde{\mathbf{X}}_{\rm t},\rm t\in\mathbb{R}^{d}_{+}\}=
\{\mathbf{X}_{\mathbf{T}_{{\rm r}'}+{\rm t}}+{\rm r}',\,{\rm t}\in\mathbb{R}_+^d\}$ is independent of 
$\{\mathbf{X}_{\rm t},\,{\rm t}\leq {\mathbf{T}_{{\rm r}'}}\}$ and has the same law as 
$\{\mathbf{X}_{\rm t},\,{\rm t}\in\mathbb{R}_+^d\}$. We conclude by noticing that  
$\tilde{\mathbf{T}}_{{\rm r}}=\inf\{\rm t : \tilde{\mathbf{X}}_{t}=-{\rm r}\}=\mathbf{T}_{{\rm r}+{\rm r}'}-\mathbf{T}_{{\rm r}'}$.\\

Assertion $3.$ follows from Lemma \ref{1377} and (\ref{3683}). Indeed, if there exists ${\rm r}\in (0,\infty)^{d}$ such that 
$\mathbb{P}(\mathbf{T}_{\rm r}\in\mathbb{R}^{d}_{+})>0$ then from Lemma \ref{1377}, for all $\bar{{\rm r}}\leq {\rm r}$, 
$\mathbf{T}_{\bar{{\rm r}}}\leq \mathbf{T}_{\rm r}$ a.s. and in particular, 
$\mathbb{P}(\mathbf{T}_{\bar{{\rm r}}}\in\mathbb{R}^{d}_{+})>0$. On the other hand, for all ${\rm r}'\in(0,\infty)^{d}$, identity 
(\ref{3683}) implies that $\mathbf{T}_{{\rm r}'}\ed  
\mathbf{T}^{(1)}_{{\rm r}^{(1)}}+...+\mathbf{T}^{(p)}_{{\rm r}^{(p)}}$ where $p\geq 1$, the ${\rm r}^{(i)}$'s are such that 
${\rm r}^{(i)}\leq \rm r$ for all $i\in[p]$, ${\rm r}^{(1)}+...+{\rm r}^{(p)}={\rm r}'$, and the $\mathbf{T}^{(i)}$'s are 
independent copies of $\mathbf{T}$. As a consequence, we obtain $\mathbb{P}(\mathbf{T}_{{\rm r}'}\in\mathbb{R}^{d}_{+})=
\prod\limits_{i=1}^{p}\mathbb{P}(\mathbf{T}^{(i)}_{{\rm r}^{(i)}}\in\mathbb{R}^{d}_{+})>0$. Now let us  prove the second part 
of this assertion. Let ${\rm r}\in(0,\infty)^{d}$ be such that $\mathbb{P}(\mathbf{T}_{\rm r}\in\mathbb{R}^{d}_{+})>0$ and let
$\lambda\in\mathbb{R}^{d}_{+}$, then by (\ref{3683}), for all ${\rm r}'\in(0,\infty)^{d}$, 
\begin{eqnarray*}
0<f(\lambda,{\rm r}+{\rm r}')&=&\mathbb{E}[e^{-\langle \lambda,\mathbf{T}_{{\rm r}+{\rm r}'}\rangle}]\\
&=&\mathbb{E}[e^{-\langle \lambda, \mathbf{T}_{{\rm r}}\rangle}]\mathbb{E}[e^{-\langle \lambda, \mathbf{T}_{{\rm r}'}\rangle}]
=f(\lambda,{\rm r})f(\lambda,{\rm r}')\,.
\end{eqnarray*}
Since $f$ is continuous and $f(\lambda,0)=\p(\mathbf{T}_{\rm r}\in\mathbb{R}^d_+)>0$, this equation implies that 
$f(\lambda,\rm r)=e^{-\langle \phi(\lambda),
{\rm r}\rangle}$, for some $\phi(\lambda)\in\mathbb{R}^{d}$. Furthermore take ${\rm r}=r{\rm e}_i$, for some $r>0$ and 
$i\in[d]$, so that $\e[e^{-\langle\lambda,\mathbf{T}_{\rm r}\rangle}]=e^{-r\phi_i(\lambda)}$. Then from right continuity, 
$\mathbf{T}_{\rm r}>0$ almost surely, so that $f(\lambda,\rm r)<1$, for all $\lambda\in\mathbb{R}^{d}_{+}$ such that 
$\lambda>0$ and thus $\phi_i(\lambda)\in(0,\infty)$. On the other hand it is plain from (\ref{5147}), the $\phi_{j}$'s are  
concave functions for all $j\in[d]$ and $\phi$ is differentiable.
\end{proof}

\noindent Note that in $(\ref{3266})$, if for some $j\neq i$, $T_{\rm r}^{(j)}=\infty$ with positive probability on the set 
$\{T_{\rm r}^{(i)}<\infty\}$, then $X^{i,j}\equiv0$, a.s. This is due to the fact that $X^{i,j}$ are subordinators for $i\neq j$, 
therefore either $X^{i,j}\equiv0$ a.s.~or $X^{i,j}_\infty=\infty$ a.s.\\

\noindent Let us emphasize the following direct consequence of Theorem \ref{8427},
\begin{equation}\label{3789}
\p(\mathbf{T}_{\rm r}\in\mathbb{R}_+^d)=e^{-\langle\phi(0),{\rm r}\rangle}\,, 
\end{equation}
so that in particular $\p(\mathbf{T}_{\rm r}\in\mathbb{R}_+^d)=1$, for all ${\rm r}\in(0,\infty)^d$ if and only if $\phi(0)=0$.
Note also that Lemma \ref{8427} does not allow us a full description of the law of the $d$-dimensional stochastic field 
$\{\mathbf{T}_{\rm r},{\rm r}\in\mathbb{R}_+^d\}$. This is the case only when $d=1$. In particular for $d\ge2$, if ${\rm r}$ and
${\rm r}'$ are not ordered, then we do not know the joint law of $(\mathbf{T}_{\rm r},\mathbf{T}_{\rm r'})$. Moreover,
looking at part $2$.~of Proposition $\ref{8427}$, one is tempted to think that, when $d\ge2$, the field 
$\{\mathbf{T}_{\rm r},\,{\rm r}\in\mathbb{R}_+^d\}$ is a spaLf, but it is actually not the case.  Indeed from the construction of this
field, the processes $\{\mathbf{T}_{r{\rm e}_i},\,r\ge0\}$, $i\in[d]$ are clearly not independent. However, it is easy to derive from 
Proposition $\ref{8427}$, that each of these processes is a multivariate subordinator whose Laplace exponent is $\phi_i$.
The following result provides an expression of its L\'evy measure. It is a consequence of further results 
(e.g. Theorem \ref{3492}) and it will be proved at the end of this paper. 

\begin{corollary}\label{8466}
Assume that $\p(\mathbf{T}_{\rm r}\in\mathbb{R}_+^d)>0$ for all ${\rm r}\in(0,\infty)^d$.
Then for all $i\in[d]$, the process $\{\mathbf{T}_{r{\rm e}_i},\,r\ge0\}$ is a multivariate subordinator whose Laplace exponent is 
$\phi_i$ given in $(\ref{5147})$. Assume moreover for all $j\in[d]$, 
the $j$-th column ${\rm X}_{\rm t}^{(j)}$ of the matrix $\mathbb{X}_{\rm t}$ admits a density 
which is continuous on $F_1\times F_2\times\dots\times F_d$, where $F_i=\mathbb{R}_+$, for $i\neq j$ 
and $F_j=\mathbb{R}$. Define the matrix $\widehat{\mathbb{X}}_{\rm t}=(\widehat{X}^{i,j}_{t_j})_{i,j\in[d]}$ by
$\widehat{X}^{i,i}_{t_i}=\sum\limits_{j=1}^{d} X^{i,j}_{t_j}$ and $\widehat{X}^{i,j}_{t_j}=X^{i,j}_{t_j}$, $i\neq j$,
and let $p_{\rm t}:M_d(\mathbb{R})\rightarrow \mathbb{R}$ be the density of $\widehat{\mathbb{X}}_{\rm t}$.
Then the L\'evy measure of the multivariate subordinator $\{\mathbf{T}_{r{\rm e}_i},\,r\ge0\}$ is given by 
\[\nu_{i}({\rm dt})
	= \int\limits_{\mathbb{R}^{d(d-1)}_{+}}
\frac{\det(-\overline{\textsc{x}}^{i,i})}{t_{1}\dots t_{d}} p_t(\textsc{x}^{0})\prod_{k\neq j}{\rm d}x_{kj}{\rm dt},\;\;\mbox{if $d>1$ and}\;\;
\nu({\rm d}t)=\frac{p_t(0)}t{\rm d}t,\;\;\mbox{if $d=1$}.\]
Here $\overline{\textsc{x}}^{i,i}$ is the matrix $\overline{\textsc{x}}=(\overline{x}_{i,j})_{i,j\in[d]}$ given by 
$\overline{x}_{i,i}=-\sum\limits_{j\neq i} x_{i,j}$ and $\overline{x}_{i,j}=x_{i,j}$ for $i\neq j$ to which line and column of index $i$
have been removed and $\textsc{x}^{0}=(x^0_{i,j})_{i,j\in[d]}$, where $x^0_{i,j}=x_{i,j}$, for $i\neq j$ and $x^0_{i,i}=0$.
\end{corollary}

We will now define a $d$-dimensional L\'evy process whose law is obtained from the law of ${\rm X}^{(j)}$ through the Esscher 
transform associated to the martingale 
\[(e^{-\langle\mu^{(j)}, {\rm X}^{(j)}_t\rangle-t\varphi_j(\mu^{(j)})})_{t\geq 0}\,,\]
for any $\mu^{(j)}\in\mathbb{R}_+^d$. Recall that $(\mathcal{F}_t^{(j)})_{t\geq 0}$ denotes the natural filtration generated by 
${\rm X}^{(j)}$. Then for $t\ge0$ and $A\in\mathcal{F}_t^{(j)}$, the law of this new L\'evy process is defined by 
\[\p^{\mu^{(j)}}(A)=\e[\ind_{A}e^{-\langle\mu^{(j)}, {\rm X}^{(j)}_t\rangle-t\varphi_j(\mu^{(j)})}]\,.\]
Let us now consider $d$ independent L\'evy processes ${\rm X}^{\mu^{(j)},(j)}$, $j\in[d]$ with respective laws $\p^{\mu^{(j)}}$.
The Laplace exponent of ${\rm X}^{\mu^{(j)},(j)}$ is given by 
\[\varphi_{j}^{\mu^{(j)}}(\lambda)=\varphi_j(\lambda+\mu^{(j)})-\varphi_j(\mu^{(j)})\,,\;\;\;\lambda\in\mathbb{R}_+^d\,.\]
Moreover, a new spaLf is obtained by setting
\begin{equation}\label{3489}
\mathbf{X}_{\rm t}^\mu:={\rm X}_{t_1}^{\mu^{(1)},(1)}+\dots+{\rm X}^{\mu^{(d)},(d)}_{t_d}\,,\;\;{\rm t}=
(t_1,\dots,t_d)\in\mathbb{R}_+^d,
\end{equation}
where $\mu=(\mu^{(1)},\dots,\mu^{(d)})\in M_d(\mathbb{R}_+)$ is the matrix whose columns are equal to $\mu^{(j)}$, $j\in[d]$. 
Let us set $\mathcal{F}_{\rm t}=\sigma\{\mathbf{X}_{\rm s},\,{\rm s}\le {\rm t}\}$ for all ${\rm t}\in\mathbb{R}^{d}_{+}$, then
$\mathcal{F}_{\rm t}=\sigma(\mathcal{F}_{t_1}^{(1)}\cup \mathcal{F}_{t_2}^{(2)}\dots\cup\mathcal{F}_{t_d}^{(d)})$ and the law of 
the spaLf $\mathbf{X}^\mu$ is given by,
\begin{equation}\label{3790}
\p^{\mu}(A)=\e[\ind_{A}e^{-\llangle\mu, \mathbb{X}_{\rm t}\rrangle-\langle{\rm t},\bar{\varphi}(\mu)\rangle}],\;\;\;{\rm t}\in
\mathbb{R}^d_+,\;\;\;A\in\mathcal{F}_{\rm t},
\end{equation}
where we have set $\bar{\varphi}(\mu)=(\varphi_1(\mu^{(1)}),\dots,\varphi_d(\mu^{(d)}))$ and we recall that 
$\llangle\mu,\mathbb{X}_{\rm t}\rrangle = \sum\limits_{j\in[d]} 
\langle\mu^{(j)},{\rm X}^{(j)}_{t_{j}}\rangle$. We will refer to (\ref{3790}) as the Esscher transform of the additive field $\mathbf{X}$. 
The Laplace exponent of $\mathbf{X}^\mu$ is then
\[\varphi^\mu(\lambda):=(\varphi_{1}^{\mu^{(1)}}(\lambda),\dots,\varphi_{d}^{\mu^{(d)}}(\lambda))\,,\;\;\;\lambda\in\mathbb{R}_+^d\,.\]    

Let us denote by $J_{\varphi}(\lambda)$, $\lambda\in (0,+\infty)^{d}$, the transpose of the opposite of the Jacobian matrix of 
$\varphi$, that is 
\begin{equation}\label{8256}
J_\varphi(\lambda)_{i,j}=-\frac{\partial}{\partial \lambda_i}\varphi_j(\lambda)\,,\;\;\;i,j\in[d]\,.
\end{equation}
Recall that since all processes $X^{i,j}$, $i,j\in[d]$, are spectrally positive L\'evy processes, their expectation is always
defined and $\e[X^{i,j}_1]\in(-\infty,\infty]$. Moreover $\varphi$ is differentiable on $(0,\infty)^d$ and the partial derivatives 
of $\varphi$ at 0 satisfy $\e[X^{i,j}_1]=-\lim\limits_{\lambda\rightarrow0}\frac{\partial}{\partial \lambda_i}\varphi_j(\lambda)$. 
We will set 
$\frac{\partial}{\partial \lambda_i}\varphi_j(0):=\lim\limits_{\lambda\rightarrow0}\frac{\partial}{\partial \lambda_i}\varphi_j(\lambda)$,
and 
\begin{equation}\label{8059}
J_\varphi(0)_{i,j}=-\frac{\partial}{\partial \lambda_i}\varphi_j(0)=\e[X^{i,j}_1]\,,\;\;\;i,j\in[d]\,.
\end{equation}
Then let us consider the following hypothesis:
\[(H)\;\;\;\;\;\mbox{The set $D:=\{\lambda\in\mathbb{R}_+^d:\varphi_j(\lambda)>0,\,j\in[d]\}$ is non empty.}\]
This hypothesis implies in particular that none of the processes $X^{j,j}$, $j\in[d]$ is a subordinator but it is actually stronger 
as we will see later on. Moreover since all $X^{i,j}$, $i\neq j$ are subordinators, it is clear that actually $D\subset(0,\infty)^d$.

\begin{theorem}\label{4526} Let ${\rm r}=(r_1,\dots,r_d)\in\mathbb{R}_+^d$ and let 
$\mathbf{T}_{\rm r}=(T_{\rm r}^{(1)},\dots,T_{\rm r}^{(d)})\in\overline{\mathbb{R}}_+^d$ be the first hitting time of level $-{\rm r}$ 
by the spaLf $\mathbf{X}$, then
\begin{itemize}
\item[$1.$]  $\mathbf{T}_{\rm r}\in{\mathbb{R}}_+^d$ holds with positive probability for some $($and hence for all$)$ 
${\rm r}\in\mathbb{R}_+^d$ if and only if $(H)$ holds.
\item[$2.$] Suppose that $(H)$ holds, then $\phi(\lambda)\in D$, for all $\lambda\in(0,\infty)^d$. Moreover, the mapping 
$\phi:(0,\infty)^d\rightarrow D$ is a diffeomorphism whose inverse corresponds to the mapping $\varphi:D\rightarrow(0,\infty)^d$, 
that is
\[\varphi(\phi(\lambda))=\lambda\,, \;\; \lambda\in(0,\infty)^d.\]
\end{itemize}
\end{theorem}
\begin{proof}
Assume that $(H)$ holds, let $\mu\in D$ and let us consider the spaLf $\mathbf{X}^\mu$ whose law is defined 
in (\ref{3790}). In the present case, $\mu$ also denotes the matrix whose each column is equal to $\mu$. Then as already 
observed, $\mu\in(0,\infty)^d$, so that all the random variables $X^{\mu,i,j}_1$ are integrable and the mean matrix of 
$\mathbf{X}^\mu$ is given by 
\[\e[X^{\mu,i,j}_1]=-\frac{\partial}{\partial\lambda_i}\varphi_j(\mu)\,,\;\;\;i,j\in[d]\,.\]
It is actually the transpose of the opposite of the Jacobian matrix of $\varphi$ denoted by $J_\varphi(\mu)$ and defined in 
(\ref{8256}). Note that $J_\varphi(\mu)$ is an essentially nonnegative matrix so that from Lemma A.2 in \cite{bp}, there is a real 
eigenvalue $\rho^\mu$ such that $\mbox{Re}(\rho)<\rho^\mu$ for all the other eigenvalues $\rho$. Moreover, since $\varphi_j$ 
is a differentiable convex function and $\varphi_j(0)=0$, one has 
\[\sum_{i=1}^d \frac{\partial}{\partial\lambda_i}\varphi_j(\mu)\mu_i\ge \varphi_j(\mu)>0\,,\]
so that from Theorem 3 of \cite{ar}, $J_\varphi(\mu)^{T}$, and therefore $J_\varphi(\mu)$, is a stable matrix in the sense of 
\cite{ar}. In particular, $\rho^\mu<0$. 

Let us first assume that $J_\varphi(\mu)$ is irreducible. Then from Lemma A.3 in \cite{bp}, we can choose an 
eigenvector ${\rm v}^\mu$ associated to $\rho^\mu$ such that $v^\mu_i>0$, for all $i\in[d]$. From the law of large numbers 
of L\'evy processes, we obtain
\[\lim_{t\rightarrow+\infty}t^{-1}\mathbf{X}^{\mu}_{t{\rm v}^\mu}=\rho^\mu {\rm v}^\mu\;\;\;a.s.\]
Therefore, from part 3.~of Lemma \ref{1377}, $\{\mathbf{X}^{\mu}_{\rm t},\,{\rm t}\in\mathbb{R}_+^d\}$ reaches each level 
$\alpha {\rm v}^\mu$, with $\alpha<0$, almost surely. Then from the definition (\ref{3790}) of the law of $\mathbf{X}^{\mu}$, 
the field $\{\mathbf{X}_{\rm t},\,{\rm t}\in\mathbb{R}_+^d\}$ reaches each level $\alpha {\rm v}^\mu$, $\alpha<0$,  with positive 
probability and since $v^\mu_i>0$, $i\in[d]$, from part 2.~of Lemma \ref{1377}, it reaches each level $-{\rm r}\in\mathbb{R}_-^d$ 
with positive probability.

Now let us assume that $J_\varphi(\mu)$ is not irreducible that is there exists a permutation matrix $P_{\sigma}$ and three 
matrices $A_{1}, A_{2}$ and $B$ such that $A_{1}$ is of size $1\leq p\leq d-1$ and
\[P_{\sigma}^{-1} J_{\varphi(\mu)} P_{\sigma} 
	=\left(\begin{matrix}
	A_{1} & 0
	\\ B & A_{2}
	\end{matrix}\right).\]
In particular, for all $(i,j)\in I\times J$ where $I=\{\sigma(1),...,\sigma(p)\}$ and $J=\{\sigma(p+1),...,\sigma(d)\}$, 
\[\mathbb{E}[X^{\mu,i,j}_{1}]=0 \; \text{ that is } \; X^{\mu,i,j}_{1}=0 \; \text{ a.s.} \] 
Therefore we can write for all ${\rm r}\in\mathbb{R}^{d}_{+}$,
\begin{eqnarray*}
\p(\mathbf{T}^{\mu}_{\rm r}\in\mathbb{R}^{d}_{+})&=&\p\left(\exists {\rm t}\in\mathbb{R}^{d}_{+}: \forall i \in[d], 
\sum\limits_{j=1}^{d} X^{\mu,i,j}(t_{j})=-r_{i}\right)\\
&=&\p\left(\exists {\rm t}\in\mathbb{R}^{d}_{+}: \forall i\in I, \sum\limits_{j\in I} X^{\mu,i,j}(t_{j})=-r_{i} \right.\\
&&\left.\text{ and } \forall i\in J, \sum\limits_{j\in J} X^{\mu,i,j}(t_{j})=-\left(r_{i}+\sum\limits_{j\in I} X^{\mu,i,j}(t_{j})\right)\right).\\
\end{eqnarray*}
Let $\mathbf{T}_{\rm r}^{\mu,I}$ be the smallest solution of the system $({\rm r}_{I},\mathbb{X}^{I,\mu})$, where we set 
${\rm r}_{I}=(r_{i})_{i\in I}$ and $\mathbb{X}^{I,\mu}=(X^{\mu,i,j})_{i,j\in I}$. Then conditioning on the event 
$\{\mathbf{T}_{\rm r}^{\mu,I}\in\mathbb{R}^{p}_{+}\}$, we obtain
\[\p(\mathbf{T}^{\mu}_{\rm r}\in\mathbb{R}^{d}_{+})
=\p(\mathbf{T}^{\mu,J}_{{\rm r}'}\in\mathbb{R}^{d-p}_{+}|\mathbf{T}_{\rm r}^{\mu,I}\in\mathbb{R}^{p}_{+})
\mathbb{P}(\mathbf{T}^{\mu,I}_{\rm r}\in\mathbb{R}^{p}_{+})\,,\]
where we have set ${\rm r}'= \left(r_{i}+\sum\limits_{j\in I} X^{\mu,i,j}(T^{\mu,I,j}_{\rm r})\right)_{i\in J}$. Then 
$\mathbf{T}_{{\rm r}'}^{\mu,J}$ is the smallest solution of the system $({\rm r}',\mathbb{X}^{J,\mu})$ with 
$\mathbb{X}^{J,\mu}=(X^{\mu,i,j})_{i,j\in J}$. Thus if $A_{1}$ and 
$A_{2}$ are irreducible, then we derive from the previous case that $\mathbb{P}(\mathbf{T}_{\rm r}^{\mu,I}\in\mathbb{R}^{p}_{+})=1$ 
and $\mathbb{P}(\mathbf{T}_{{\rm r}'}^{\mu,J}\in\mathbb{R}^{d-p}_{+}|\mathbf{T}_{\rm r}^{\mu,I}\in\mathbb{R}^{p}_{+})=1$. In other 
words, we have $\mathbb{P}(\mathbf{T}_{\rm r}^{\mu}\in\mathbb{R}^{d}_{+})=1$ and then 
$\mathbb{P}(\mathbf{T}_{\rm r}\in\mathbb{R}^{d}_{+})>0$. On the other hand, if $A_{1}$ and/or $A_{2}$ are not irreducible, then
we can repeat this argument.\\

Conversely, let us assume that  $\mathbf{T}_{\rm r}\in{\mathbb{R}}_+^d$ holds with positive probability for all 
${\rm r}\in\mathbb{R}_+^d$ and let $\phi$ be the function defined in part 3 of Proposition \ref{8427}. Let us show that for 
all $\lambda\in(0,\infty)^d$, $\varphi(\phi(\lambda))=\lambda$, which implies in particular that $\phi(\lambda)\in D$. 
It follows from the independence and stationarity of the increments of the spaLf $\{{\bf X}_{\rm t},\,{\rm t}\in\mathbb{R}_+^d\}$ 
that for all ${\rm r},{\rm t},\lambda\in\mathbb{R}_+^d$,
\begin{eqnarray*}
\e[e^{-\langle\lambda, {\bf T}_{\rm r}\rangle}\ind_{\{{\rm t}<{\bf T}_{\rm r}\}}]&=&\int_{C_{\rm r}}
\e[e^{-\langle\lambda, {\bf T}_{\rm r}\rangle}\ind_{\{{\rm t}<{\bf T}_{\rm r}\}}\,|\,{\bf X}_{\rm t}={\rm x}]\p({\bf X}_{\rm t}\in d{\rm x})\\
&=&\int_{C_{\rm r}} e^{-\langle\lambda, {\rm t}\rangle}
\e[e^{-\langle\lambda, {\bf T}_{\rm r+x}\rangle}]\p({\bf X}_{\rm t}\in d{\rm x})\\
&=&e^{-\langle\lambda, {\rm t}\rangle}e^{-\langle {\rm r},\phi(\lambda)\rangle}
\left[e^{\langle\varphi(\phi(\lambda)),{\rm t}\rangle}-\int_{-\infty}^{-r_1}\dots \int_{-\infty}^{-r_d}
e^{-\langle {\rm x},\phi(\lambda)\rangle}\p({\bf X}_{\rm t}\in d{\rm x})\right],
\end{eqnarray*}
where $C_{\rm r}$ is the union of all the sets $E_{1}\times\dots\times E_{d}$ with at least one $i\in[d]$ such that 
$E_{i}=]-r_{i},+\infty[$ and for the others $j\in[d]$, $E_{j}=\mathbb{R}$. Then we derive the identity
\begin{equation}\label{2511}
1-e^{\langle {\rm r},\phi(\lambda)\rangle}\e[e^{-\langle\lambda, {\bf T}_{\rm r}\rangle}\ind_{\{{\rm t}<{\bf T}_{\rm r}\}^c}]=
e^{-\langle\lambda, {\rm t}\rangle}
\left[e^{\langle\varphi(\phi(\lambda)),{\rm t}\rangle}-\int_{-\infty}^{-r_1}\dots \int_{-\infty}^{-r_d}
e^{-\langle {\rm x},\phi(\lambda)\rangle}\p({\bf X}_{\rm t}\in d{\rm x})\right].
\end{equation}
Let ${\rm r}',{\rm r}''\in(0,\infty)^d$ be such that ${\rm r}'+{\rm r}''={\rm r}$, then from Proposition \ref{8427}, ${\bf T}_{\rm r}$ can 
be decomposed as ${\bf T}_{\rm r}={\bf T}_{{\rm r}'}+\tilde{{\bf T}}_{{\rm r}''}$, where $\tilde{{\bf T}}_{{\rm r}''}$ is an independent 
copy of ${\bf T}_{{\rm r}''}$. Moreover $\{{\rm t}<{\bf T}_{\rm r}\}^c\subset \{{\rm t}<{\bf T}_{{\rm r}'}\}^c\cap\{{\rm t}<
\tilde{{\bf T}}_{{\rm r}''}\}^c$, so that 
\[\e[e^{-\langle\lambda, {\bf T}_{\rm r}\rangle}\ind_{\{{\rm t}<{\bf T}_{\rm r}\}^c}]\le \e[e^{-\langle\lambda, {\bf T}_{{\rm r}'}\rangle}
\ind_{\{{\rm t}<{\bf T}_{{\rm r}'}\}^c}]\e[e^{-\langle\lambda,{\bf T}_{{\rm r}''}\rangle}\ind_{\{{\rm t}<{\bf T}_{{\rm r}''}\}^c}].\]
If the coordinates of ${\rm r}$ are integers, then applying this identity recursively, we obtain,
\begin{equation}\label{3733}\e[e^{-\langle\lambda, {\bf T}_{\rm r}\rangle}\ind_{\{{\rm t}<{\bf T}_{\rm r}\}^c}]\le\prod_{j=1}^d
\e[e^{-\langle\lambda,{\bf T}_{e_j}\rangle}\ind_{\{{\rm t}<{\bf T}_{e_j}\}^c}]^{r_j}\,.
\end{equation}
Then we can find ${\rm t}$ whose coordinates are sufficiently small so that for all $j$, 
\[\e[e^{-\langle\lambda,{\bf T}_{e_j}\rangle}\ind_{\{{\rm t}<{\bf T}_{e_j}\}^c}]<\e[e^{-\langle\lambda,{\bf T}_{e_j}\rangle}]=
e^{-\phi_j(\lambda)}.\] 
Therefore $\lim\limits_{{\rm r}\rightarrow\infty}e^{\langle {\rm r},\phi(\lambda)\rangle}
\prod_{j=1}^d\e[e^{-\langle\lambda,{\bf T}_{e_j}\rangle}\ind_{\{{\rm t}<{\bf T}_{e_j}\}^c}]^{r_j}=0$ and from (\ref{3733}) we derive 
that the left member of (\ref{2511}) tends to 1, while the right member tends to $e^{-\langle\lambda, {\rm t}\rangle}
e^{\langle\varphi(\phi(\lambda)),{\rm t}\rangle}$, which shows that $\varphi(\phi(\lambda))=\lambda$.
This is true in particular for all $\lambda\in(0,\infty)^d$ and hence $D$ is not empty. This achieves the proof of both assertions 
$1.$ and $2$.
\end{proof}
\noindent From part 1.~of Theorem \ref{4526}, assuming $(H)$ for a spaLf ${\bf X}$ ensures that ${\bf X}$ hits all negative levels 
in a finite time
with positive probability. When $d=1$, this is simply assuming that the spectrally positive L\'evy process we consider is not a 
subordinator.\\     

Let us give an example of a $2$-dimensional spaLf. Assume that, for $j\in [2]$, the $X^{j,j}$'s are independent Brownian motions 
$B^{(j)}$ with drifts $a_j\in\mathbb{R}$, that is $X^{j,j}_t=B^{(j)}_t+a_jt$ and that for $i\neq j$, $X^{i,j}$ is a pure drift, that is 
$X^{i,j}_t\equiv a_{ij}t$, $a_{ij}\ge0$. The Laplace exponents $\varphi_j$ are then explicitly given by 
\[\varphi_j(\lambda)=-\lambda_ia_{ij}-\lambda_ja_j+\frac12q_j\lambda_j^2,\;\;\; i\neq j,\lambda\in\mathbb{R}^{2}_{+}\,.\]
Assume $q_{j}>0$, $j\in[2]$. After some calculus, we are able to explicit the set $D$ defined in hypothesis $(H)$. It is given by 
\[D=\left\lbrace\lambda\in\mathbb{R}^{2}_{+} : 
\lambda_{1}>\left(\dfrac{a_{1}+\sqrt{\Delta_{1}(\lambda_{2})}}{q_{1}}\vee 0\right) \text{ and } 
\lambda_{2}>\left(\dfrac{a_{2}+\sqrt{\Delta_{2}(\lambda_{1})}}{q_{2}}\vee 0\right)\right\rbrace,\] 
where $\Delta_{j}(\lambda_{i})=
a_{j}^{2}+2a_{ij}q_{j}\lambda_{i}$ for all $j\in[2]$ and $i\neq j$. Note that this set is not empty and so the assumption $(H)$ holds. 
In particular, thanks to Theorem \ref{4526}, the spaLf $\textbf{X}$ reaches all the level $-{\rm r}\in \mathbb{R}^{2}_{-}$ with 
positive probability and according to the second part of this theorem, we know that the mapping $\varphi$ admits an inverse $\phi$ 
on the set $D$. This inverse $\phi=(\phi_{1},\phi_{2})$ is given by
\[\phi_{j}(\lambda)=\dfrac{1}{q_{j}}\sqrt{2q_{j}\lambda_{j} + a_{j}^{2}+2a_{ij}q_{j}\phi_{i}(\lambda)} +\dfrac{a_{j}}{q_{j}},\;\;\; j\in[2],
i\neq j,\lambda\in\mathbb{R}^{2}_{+}\,.\] \\

In order to carry on with the general study of the fluctuation of the spaLf $\mathbf{X}$, we shall now give a characterization of 
the condition $\phi(0)=0$ in terms of the Jacobian matrix $J_\varphi(0)$. As a first remark, 
note that if for some $j\in[d]$, $J_\varphi(0)_{j,j}>0$, then $\lim\limits_{t\rightarrow+\infty}X^{j,j}_t=+\infty$ a.s. and hence the field
$\{\mathbf{X}_{\rm t},\,{\rm t}\in\mathbb{R}_+^d\}$ cannot reach all the levels $-{\rm r}\in\mathbb{R}_-^d$ with probability one. 
Therefore $\phi(0) >0$ whenever there is $j$ such that $J_\varphi(0)_{j,j}>0$.\\

Recall that whenever the essentially nonnegative matrices $J_{\varphi}(\lambda)$, defined in $(\ref{8256})$ and $(\ref{8059})$
for $\lambda\in[0,\infty)^d$ have finite entries and are irreducible, according to the Perron-Frobenius theory, there 
are real eigenvalues $\rho^\lambda$ with multiplicity equal to 1 and such that the real part of any other eigenvalue is 
less than $\rho^\lambda$, see Appendix A of \cite{bp}. We set $\rho^0=\rho$.

\begin{theorem}\label{9503} Assume that $(H)$ holds and that $J_{\varphi}(0)$ is irreducible, then 
\begin{itemize}
\item[$1.$] the values $0$ and $\phi(0)$ are the only roots of the equation $\varphi(\lambda)=0$, $\lambda\in\mathbb{R}_+^d$. 
Furthermore, either $\phi(0)$ is equal to $0$ or it belongs to $(0,\infty)^{d}$.
\item[$2.$]
If $\e[X_1^{i,j}]=\infty$, for some $i,j\in[d]$, then $\phi(0)>0$. Assume that 
$\e[X_1^{i,j}]<\infty$, for all $i,j\in[d]$, then $\phi(0)=0$ if and only if $\rho\le0$.
\end{itemize}
\end{theorem}
\begin{proof} Let us assume that $J_{\varphi}(0)$ is irreducible. Since $\varphi:D\rightarrow (0,\infty)^{d}$ is the inverse 
of $\phi:(0,\infty)^{d}\rightarrow D$, $\phi(0)$ is the only solution of the equation $\varphi(\lambda)=0$ on $\overline{D}$. 
Indeed, let $\mu\in\overline{D}$ such that $\varphi(\mu)=0$ and $\mu_n\in D$ such that 
$\lim\limits_{n\rightarrow +\infty}\mu_n=\mu$. Then by continuity, $\lim\limits_{n\rightarrow +\infty}\varphi(\mu_n)=0$ and 
$\Phi(0)=\lim\limits_{n\rightarrow +\infty}\phi(\varphi(\mu_n))=\lim\limits_{n\rightarrow +\infty}\mu_n$, so that $\mu=\Phi(0)$.

Now let $\mu\in\mathbb{R}^{d}_{+}\setminus\{0,\phi(0)\}$ be a solution of the equation $\varphi(\lambda)=0$ and ${\rm u}=
\dfrac{\mu}{||\mu||}$. Then we consider, for all $j\in[d]$, the function $f_{j}:a\in\mathbb{R}\mapsto\varphi_{j}(\mu+a{\rm u})$. 
Let us first note that since $\varphi_{j}$ is convex, so is $f_{j}$. 
Furthermore, for all $j\in[d]$, we have $f_{j}(0)=\varphi_{j}(\mu)=0=\varphi_{j}(0)=f_{j}(-||\mu||)$. On the one hand, if there exists 
$j\in[d]$ such that $\mu_{j}=0$, then for all $a\in\mathbb{R}$, $\mu_{j}+au_{j}=0$ that is $f_{j}(a)=\varphi_{j}(\mu+a{\rm u})\leq 0$. 
Since $0$ and $-||\mu||<0$ are zeros of the real convex function $f_{j}$, it implies that $f_{j}$ is constant equal to $0$. 
In other words, for all $t\geq 0$, 
\[\mathbb{E}\left[e^{-\sum\limits_{i\neq j} (\mu_{i}+au_{i})X^{i,j}_{t}}\right]=e^{t\varphi_{j}(\mu +a{\rm u})}=1\]
and then for all $i\in[d]$, $X^{i,j}\equiv0$ a.s. that is $J_{\varphi}(0)$ is reducible. Since we assumed $J_{\varphi}(0)$ irreducible, 
we necessarily have $\mu_{j}>0$, $j\in[d]$ and then, by convexity, $f_{j}$ is negative on $(-||\mu||,0)$ and positive on $(0,+\infty)$. 
In other words, for all integers $j\in[d]$ and for all $\epsilon>0$, $\varphi_{j}(\mu+\epsilon {\rm u})>0$ that is $\mu\in\overline{D}$ 
which is a contradiction. As a consequence, when $J_{\varphi}(0)$ is irreducible, there is at most two solutions of the equation 
$\varphi(\lambda)=0$, $\lambda\in\mathbb{R}^{d}_{+}$ which are $0$ and $\phi(0)\in \overline{D}$. Furthermore, when
$J_{\varphi}(0)$ is irreducible, we have seen that $\phi(0)=0$ or $\phi(0)\in(0,\infty)^{d}$.\\

Let us now prove assertion 2. Suppose that $\e[X_1^{i,j}]=\infty$, for some $i,j\in[d]$. Then for all $\lambda\in(0,\infty)^d$ small 
enough, $\varphi_j(\lambda)<0$. Indeed, let $\lambda\in(0,\infty)^d$. Since the spectrally positive L\'evy process 
$\langle \lambda,X_t^{(j)}\rangle$ drifts to $\infty$, for all $\alpha\in(0,\infty)$ small enough, its characteristic exponent 
evaluated at $\alpha$ is negative, that is $\varphi_j(\alpha\cdot\lambda)<0$. But if $\phi(0)=0$, since $0\in\overline{D}$,  
there is $\lambda\in(0,\infty)^d$ small enough such that $\varphi_j(\lambda)>0$. Therefore, $\phi(0)>0$.

Suppose now that $\e[X_1^{i,j}]<\infty$, for all $i,j\in[d]$ and that $\rho<0$. Let ${\rm u}=(u_1,\dots,u_d)$ be the unique right 
eigenvector corresponding to $\rho$ such that $u_i>0$ for all $i\in[d]$, and $u_1+\dots+u_d=1$, see Lemma A.2 in \cite{bp}. 
Then from the law of large numbers, 
\[\lim\limits_{t\rightarrow+\infty}t^{-1}\mathbf{X}_{t{\rm u}}=\rho {\rm u}\,,\;\;\;a.s.\]
Therefore, $\{\mathbf{X}_{\rm t},\,{\rm t}\in\mathbb{R}_+^d\}$ reaches a.s. all the levels $\alpha {\rm u}$, $\alpha<0$ and from 
Lemma \ref{8427} it reaches all the levels $-{\rm r}\in\mathbb{R}_-^d$ a.s. We conclude from (\ref{3789}) that $\phi(0)=0$.

Assume that $\rho=0$.  Let ${\rm u}=(u_1,\dots,u_d)$ be a right eigenvector corresponding to $\rho$, then from the law of large 
numbers, 
\[\lim\limits_{t\rightarrow+\infty}t^{-1}\mathbf{X}_{t{\rm u}}=0 \,,\;\;\;a.s.\]
Therefore, for all $i\in[d]$, the process $Y^{i}=(Y^{i}_{t})_{t\geq 0}$, defined for all $t\geq 0$, by $Y^{i}_{t}=\sum\limits_{j=1}^{d} 
X^{i,j}_{tu_{j}}$ is a real L\'evy process such that 
\[\lim\limits_{t\rightarrow+\infty}t^{-1}Y^{i}_{t}=0\,,\;\;\;a.s.\]
that is, for all $i\in[d]$, $Y^{i}$ oscillates. On the other hand, if $\phi(0)>0$, then,  by convexity of the $\varphi_{j}$'s, there exists 
$\lambda\in\mathbb{R}^{d}_{+}$ such that $\varphi_{j}(\lambda)<0$, for all $j\in[d]$. Consequently, for all direction 
${\rm v}\in\mathbb{R}^{d}_+$, we have 
\[ \mathbb{E}[e^{-\langle\lambda, \mathbf{X}_{t{\rm v}}\rangle}]=e^{\langle t{\rm v},\varphi(\lambda)\rangle}
\underset{t\rightarrow +\infty}{\rightarrow} 0\,.\]
It implies that for all direction ${\rm v}\in\mathbb{R}^{d}_{+}$, the L\'evy process 
$\langle \lambda,\mathbf{X}_{t{\rm v}}\rangle$  tends to $\infty$ in probability (and hence almost surely), as $t\rightarrow\infty$. 
In particular, for ${\rm v}={\rm u}$, there exists $i\in[d]$ such that $Y_{t}^{i}$ tends to $\infty$ almost surely, as $t\rightarrow\infty$, 
which is a contradiction. In conclusion, $\phi(0)=0$.\\

Conversely, assume that $\phi(0)=0$ then $0\in\overline{D}$ and by convexity, there exists $\mu\in(0,+\infty)^{d}$, small enough,
such that $\varphi_{i}(\mu)>0$, for all $i\in[d]$. Recall from  (\ref{3789}) and (\ref{3790}) the definition of the Esscher transform 
$\mathbf{X}^{\mu}$ of the spaLf $\mathbf{X}$, with $\mu^{(1)}=\dots=\mu^{(d)}=\mu$. We have seen in the proof of Theorem 
3.1 that the Perron-Frobenius eigenvalue of $J_{\varphi}(\mu)$ satisfies $\rho^{\mu}<0$. Since the $\varphi_{j}$'s are 
$\mathcal{C}^{\infty}$-functions, for all $i,j\in[d], \dfrac{\partial}{\partial \lambda_{i}}\varphi_{j}$ are continuous and hence 
$\lim\limits_{\mu\rightarrow 0} J_{\varphi}(\mu)=J_{\varphi}(0)$. Furthermore, the eigenvalues of the matrix $J_{\varphi}(\mu)$ 
depend continuously of its entries because they are the roots of its characteristic polynomial whose coefficients are polynomial 
functions of the entries of the matrix. Then since $\rho^{\mu}=\max\limits_{i\in[d]}\mbox{Re}(\lambda_{i}^{\mu})$ and 
$\rho=\max\limits_{i\in[d]}\mbox{Re}(\lambda_{i})$ where $\lambda_{i}^{\mu}$ and $\lambda_{i}$ are respectively the eigenvalues 
of $J_{\varphi}(\mu)$ and $J_{\varphi}(0)$, we have that $\lim\limits_{\mu\rightarrow 0} \rho^{\mu}=\rho\leq 0$.
\end{proof}

\noindent Assuming $(H)$, we will say that the additive  L\'evy field $(\mathbf{X}_{\rm t},{\rm t}\in\mathbb{R}_+^d)$ 
drifts to $-\infty$, oscillates or drifts to $+\infty$ according as $\rho<0$, $\rho=0$ or $\rho>0$. \\

Let us go back to our example. We already have the explicit form of $\varphi$, the set $D$ and the inverse $\phi$. 
Now we want to find the solution of the equation 
$\varphi(\lambda)=0$, $\lambda\in\mathbb{R}^{2}_{+}$. Assume $J_{\varphi}(0)$ is irreducible that is $a_{ij}>0$ for all $i\neq j$.
Then the solutions of the equation $\varphi(\lambda)=0$, $\lambda\in\mathbb{R}^{2}_{+}$ are $(0,0)$ and points of the form 
$\left(\dfrac{a_{1}+\sqrt{\Delta_{1}(\lambda_{2})}}{q_{1}},\dfrac{a_{2}+\sqrt{\Delta_{2}(\lambda_{1})}}{q_{2}}\right)$ where 
$\Delta_{j}(\lambda_{i})=a_{j}^{2}+2a_{ij}q_{j}\lambda_{i}$, $j\in[2],i\neq j$. It is easy to check that there is only one solution of 
the second kind and it is in $(0,+\infty)^{2}$ or equal to $0$. According to the expression of $\phi$, $\phi(0)$ is this solution. 
We can show that $\phi(0)=0$ if and only if $a_{1}<0$, $a_{2}<0$ and $a_{1}a_{2}\geq a_{1,2}a_{2,1}$. Furthermore, we can 
compute the Perron-Frobenius eigenvalue $\rho$ of the Jacobian $J_{\varphi}(0)$. It has the form
\[\rho = \dfrac{a_{1}+a_{2}+\sqrt{(a_{1}-a_{2})^{2} + 4a_{1,2}a_{2,1}}}{2}\,.\ \]
Then it is easy to see that $\rho\leq 0$ if and only if $a_{1}<0$, $a_{2}<0$ and $a_{1}a_{2}\geq a_{1,2}a_{2,1}$. 
In conclusion, we find again $\phi(0)=0\Leftrightarrow\rho\leq 0$.

Note that if $J_{\varphi}(0)$ is reducible then at least one of the $a_{i,j}$ is equal to zero, for $i,j\in[d]$. 
Then $\varphi$ has at most four zeros (indeed some can be equals or negative) $(0,0)$, 
$\left(\dfrac{2a_{1}}{q_{1}},0\right)$, $\left(0,\dfrac{2a_{2}}{q_{2}}\right)$ and $\left(\dfrac{a_{1}+\sqrt{\Delta_{1}
(\frac{2a_{2}}{q_{2}})}}{q_{1}},\dfrac{a_{2}+\sqrt{\Delta_{2}
(\frac{2a_{1}}{q_{1}})}}{q_{2}}\right)=\phi(0)$.

\begin{remark}
We have proved in part $1.$~of Theorem $\ref{9503}$, that if there exists a solution to the equation $\varphi(\lambda)=0$ 
in $(0,+\infty)^{d}$ then it is unique and equal to $\phi(0)$ in both cases $J_{\varphi}(0)$ reducible or irreducible. Nevertheless, 
in the reducible case, we can't say anything about the solutions $\mu\in\mathbb{R}^{d}\setminus\{0\}$ with $\mu_{j}=0$ for 
some $j\in[d]$. Furthermore, we don't know when $\phi(0)$ is such a solution nor characterize the cases when it happens.

In part $2.$~of Theorem $\ref{9503}$, we actually proved that when $\phi(0)>0$, for each direction ${\rm v}\in\mathbb{R}^{d}_{+}$, 
there is at least one coordinate of the field $\mathbf{X}$ which goes to $+\infty$, almost surely.
\end{remark}

%%%%%%%%%%%%%%%%%%%%%%%%%%%%%%%%%%%%%%%%%%%%%%%%%%%%%%%%%%%%%%%%%%%%%
\section{On the scale field $(\mathbf{T}_{\rm r},\mathbb{X}_{\mathbf{T}_{\rm r}})$}\label{6932}
%%%%%%%%%%%%%%%%%%%%%%%%%%%%%%%%%%%%%%%%%%%%%%%%%%%%%%%%%%%%%%%%%%%%%

Let us recall the definition of the matrix valued field $\mathbb{X}=\{\mathbb{X}_{\rm t}, {\rm t}\in\mathbb{R}^{d}_{+}\}$ given in the 
beginning of Section \ref{2018}. As already noticed, this field carries on the same information as the spaLf $\mathbf{X}$. 
However, whereas the vector $\mathbf{X}_{\mathbf{T}_{\rm r}}$ is deterministic on the set $\{\mathbf{T}_{\rm r}\in\mathbb{R}_+^d\}$ 
(and is actually equal to $-{\rm r}$), the matrix  $\mathbb{X}_{\mathbf{T}_{\rm r}}$ is random whenever $d\ge2$. From another point 
of view, the fact that the field ${\rm r}\mapsto(\mathbf{T}_{\rm r},\mathbb{X}_{\mathbf{T}_{\rm r}})$ has independent and stationary 
increments (see the next theorem) induces an analogy with fluctuation theory in dimension 1. This bivariate field can be 
considered as some scale field describing the fluctuations of the field $\mathbb{X}$ at its 'infimum'. The aim of this section is to 
describe the law of this scale field, first through its Laplace exponent and then from a Kemperman's type identity relating its law 
to this of the field $\mathbb{X}$.

Recall that we denote by $\mu^{(j)}$ the $j$-th column of the matrix $\mu=(\mu_{i,j})_{i,j\in[d]}$.  Then given a spaLf 
$\mathbf{X}$ we define the set 
\[\mathcal{M}_\varphi=\{(\lambda,\mu)\in\mathbb{R}^{d}_{+}\times{M}_d(\mathbb{R}_{+}):
\lambda_{j}\ge\varphi_{j}(\mu^{(j)}),\,j\in[d]\}.\]

\begin{theorem}\label{4126} Assume that $(H)$ holds. Let ${\rm r}=(r_1,\dots,r_d)\in\mathbb{R}_+^d$ and let 
$\mathbf{T}_{\rm r}$ be the first hitting time of level $-{\rm r}$ by the spaLf $\mathbf{X}$, then there exits a mapping 
$\Phi=(\Phi_{1},\dots,\Phi_{d}) : \mathcal{M}_\varphi \rightarrow \mathbb{R}^{d}_{+}$ such that
\[\mathbb{E}\left[e^{-\langle\lambda,\mathbf{T}_{\rm r}\rangle -\llangle \mu,
\mathbb{X}_{\mathbf{T}_{\rm r}}\rrangle}\ind_{\{\mathbf{T}_{\rm r}\in\mathbb{R}^d_+\}}\right]=
e^{-\langle {\rm r},\Phi(\lambda,\mu)\rangle},\;\;\; (\lambda,\mu)\in\mathcal{M}_\varphi.\]
Moreover $\Phi$ satisfies the equations, 
\begin{equation}\label{2978}
\varphi_{j}(\mu^{(j)}+
\Phi(\lambda,\mu))=\lambda_{j}, \;\;\; j\in[d],\;\;\;(\lambda,\mu)\in\mathcal{M}_\varphi,
\end{equation}
and it is explicitly determined by
\begin{equation}\label{8778}
\Phi(\lambda,\mu)=\phi^{\mu}(\lambda_{1}-\varphi_{1}(\mu^{(1)}),\dots,\lambda_{d}-\varphi_{d}(\mu^{(d)}))
\end{equation}
where $\phi^{\mu}$ is the inverse of the Laplace exponent $\varphi^{\mu}=(\varphi_{1}^{\mu^{(1)}},\dots,\varphi_{d}^{\mu^{(d)}})$ of 
the Esscher transform $\mathbf{X}^{\mu}$ defined in $(\ref{3489})$. 
\end{theorem}
\begin{proof} Let us first note that the random field $\{\mathbf{M}_{\rm t},\,{\rm t}\in\mathbb{R}_+^d\}:=
\{e^{-\langle\bar{\varphi}(\mu),{\rm t}\rangle -\llangle \mu,
\mathbb{X}_{{\rm t}}\rrangle},\,{\rm t}\in\mathbb{R}_+^d\}$, where $ \bar{\varphi}(\mu)=(\varphi_{1}^{\mu^{(1)}}(\lambda),\dots,
\varphi_{d}^{\mu^{(d)}}(\lambda))$, is a multi-indexed martingale with respect to the filtration  
$\mathcal{F}_{\rm t}=\sigma\{\mathbf{X}_{\rm s},\,{\rm s}\le {\rm t}\}=
\sigma(\mathcal{F}_{t_1}^{(1)}\cup \mathcal{F}_{t_2}^{(2)}\dots\cup\mathcal{F}_{t_d}^{(d)})$, ${\rm t}\in\mathbb{R}^{d}_{+}$ 
in the sense of \cite{ku}. Fix ${\rm r}=(r_1,\dots,r_d)\in\mathbb{R}_+^d$ and define the sequence of multivariate random 
times $\mathbf{T}_{n,\rm r}=(T_{n,\rm r}^{(1)},\dots,T_{n,\rm r}^{(d)})$, $n\ge1$ by
\[T_{n,\rm r}^{(i)}=\sum_{k\ge0}{2^{-n}(k+1)}\ind_{\{2^{-n}k\le T_{\rm r}^{(i)}<2^{-n}(k+1)\}}+\infty\cdot 
\ind_{\{T_{\rm r}^{(i)}=\infty\}}.\]
Then $\mathbf{T}_{\rm r}$ and $\mathbf{T}_{n,\rm r}$, $n\ge1$ are stopping times of the filtration 
$(\mathcal{F}_{\rm t})_{{\rm t}\in\mathbb{R}^{d}_{+}}$ in the sense of \cite{ku}. Moreover, for each $i\in[d]$, the sequence 
$(T_{n,\rm r}^{(i)})_{n\geq 1}$ is non increasing and tends to $T_{\rm r}^{(i)}$ 
almost surely. Now for all ${\rm u}\in\mathbb{R}_+^d$, define $\mathbf{T}_{n,\rm r}^{({\rm u})}$ by
\[\mathbf{T}_{n,\rm r}^{({\rm u})}=\left\{\begin{array}{ll}
\mathbf{T}_{n,\rm r}\;\;&\mbox{on}\;\;\{\mathbf{T}_{n,\rm r}\le {\rm u}\}\\
{\rm u}\;\;&\mbox{on}\;\;\{\mathbf{T}_{n,\rm r}\le {\rm u}\}^c\,.
\end{array}\right.\]
Then $\mathbf{T}_{n,\rm r}^{({\rm u})}$ is a stopping time (see for instance the proof of Lemma (2.3) in \cite{ku}). Moreover, 
\[\mathbf{M}_{\mathbf{T}_{n,\rm r}^{({\rm u})}}=\sum_{{\rm v}\in {\rm D}_n,{\rm v}\leq{\rm u}}
\mathbf{M}_{\rm v}\ind_{\{\mathbf{T}_{n,\rm r}={\rm v}\}}+\mathbf{M}_{\rm u}\ind_{\{\mathbf{T}_{n,\rm r}\le {\rm u}\}^c}\le 
\sum_{{\rm v}\in {\rm D}_n,{\rm v}\leq{\rm u}}\mathbf{M}_{\rm v}+\mathbf{M}_{\rm u},\]
where ${\rm D}_n=\{{\rm v}\in\mathbb{R}_+^d:{\rm v}=2^{-n}k,\,k\ge0\}$.
Since the set $\{{\rm v}\in {\rm D}_n,{\rm v}\leq{\rm u}\}$ is finite, $\e\left[\mathbf{M}_{\mathbf{T}_{n,\rm r}^{({\rm u})}}\right]<\infty$. 
Moreover $\mathbf{T}_{n,\rm r}^{({\rm u})}$ and $\mathbf{M}_{\rm u}$ clearly satisfy the conditions (2.4) and (2.5) of Lemma (2.3) 
in \cite{ku}. Therefore, in virtue of this lemma,
\[\e\left[\mathbf{M}_{\mathbf{T}_{n,\rm r}^{({\rm u})}}\right]=1.\]
Then $\lim\limits_{n\rightarrow\infty}\mathbf{T}_{n,\rm r}^{({\rm u})}=\mathbf{T}_{\rm r}^{({\rm u})}$ almost surely, where 
\[\mathbf{T}_{\rm r}^{({\rm u})}=\left\{\begin{array}{ll}
\mathbf{T}_{\rm r}\;\;&\mbox{on}\;\;\{\mathbf{T}_{\rm r}\le {\rm u}\}\\
{\rm u}\;\;&\mbox{on}\;\;\{\mathbf{T}_{\rm r}\le {\rm u}\}^c\,,
\end{array}\right.\]
so that by Fatou's Lemma and right continuity of $\{ \mathbf{M}_{\rm t},\,{\rm t}\in\mathbb{R}_+^d\}$, we obtain as $n$ tends 
to $\infty$, $\e\left[\mathbf{M}_{\mathbf{T}_{\rm r}^{({\rm u})}}\right]\le1$. Then by applying Fatou's Lemma again, we obtain as 
each coordinate of ${\rm u}$ tends to $\infty$ that 
$\e\left[\mathbf{M}_{\mathbf{T}_{\rm r}}\ind_{\{\mathbf{T}_{\rm r}\in\mathbb{R}^d_+\}}\right]\le1$. It implies
that for all $(\lambda,\mu)\in\mathcal{M}_\varphi$, $\mathbb{E}\left[e^{-\langle\lambda,\mathbf{T}_{\rm r}\rangle -\llangle \mu,
\mathbb{X}_{\mathbf{T}_{\rm r}}\rrangle}\ind_{\{\mathbf{T}_{\rm r}\in\mathbb{R}^d_+\}}\right]\le1$.

Then we prove in the same way as for (\ref{3683}) in Proposition \ref{8427}, that for all 
${\rm r},{\rm r}'\in\mathbb{R}^d_+$, 
\begin{equation}\label{1087}
(\mathbf{T}_{{\rm r}+{\rm r}'},\mathbb{X}_{\mathbf{T}_{{\rm r}+{\rm r}'}})\ind_{\{\mathbf{T}_{{\rm r}+{\rm r}'}\in\mathbb{R}^d_+\}}
\ed (\mathbf{T}_{\rm r}+ \mathbf{T}'_{{\rm r}'},\mathbb{X}_{\mathbf{T}_{\rm r}}+\mathbb{X}'_{\mathbf{T}'_{{\rm r}'}})
\ind_{\{\mathbf{T}_{\rm r}+\mathbf{T}'_{{\rm r}'}\in\mathbb{R}^d_+\}},
\end{equation}
where $\mathbb{X}'$ is an independent copy of $\mathbb{X}$ and $\mathbf{T}'$ is its first hitting time process.
Recall that under assumption $(H)$,
$\p(\mathbf{T}_{\rm r}\in\mathbb{R}^d_+)>0$ for all ${\rm r}\in\mathbb{R}^d_+$.
The existence of the mapping $\Phi$ follows by using (\ref{1087}), in the same way as for the existence of the mapping $\phi$ 
in 3.~of Proposition \ref{8427}. (Note that in particular $\Phi(\lambda,0)=\phi(\lambda)$, $\lambda\in\mathbb{R}^d_+$.)\\

Then it is readily seen that 
\begin{equation}\label{6587}
(\mathbf{T}_{\rm r},\mathbb{X}_{\mathbf{T}_{\rm r}}) = ({\rm r},\mathbb{X}_{\rm r})+(\tilde{\mathbf{T}}_{{\rm r}+
\mathbf{X}_{\rm r}},\tilde{\mathbb{X}}_{\tilde{\mathbf{T}}_{{\rm r}+\mathbf{X}_{\rm r}}})\;\;
\mbox{a.s. on $\{\mathbf{T}_{\rm r}\in\mathbb{R}^d_+\}$,}
\end{equation}
where $\tilde{\mathbb{X}}_{\rm t}=\mathbb{X}_{\rm r+t}-\mathbb{X}_{\rm r}$ and 
$\tilde{\mathbf{T}}_{\rm k}=\inf\{{\rm t}\geq 0 :\tilde{\mathbf{X}}_{\rm t}=-{\rm k}\}$. Since $\mathbf{X}$ is a spaLf, for all 
${\rm t}\in\mathbb{R}^{d}_{+}$, $\tilde{\mathbf{X}}_{\rm t}$ has the same law as $\mathbf{X}_{\rm t}$ and is independent of 
$\{ \mathbf{X}_{\rm s} : {\rm s}\leq {\rm r}\}$. Thus conditionally on $\{\mathbf{T}_{\rm r}\in\mathbb{R}^d_+\}$, 
$\tilde{\mathbf{T}}_{{\rm r}+\mathbf{X}_{\rm r}}$ 
and $\tilde{\mathbb{X}}_{\tilde{\mathbf{T}}_{{\rm r}+\mathbf{X}_{\rm r}}}$ are independent of $\mathbb{X}_{\rm r}$. 
Let $(\lambda,\mu)\in\mathcal{M}_\varphi$, then using (\ref{6587}), we obtain
\[e^{-\langle r, \Phi(\lambda,\mu)\rangle}
	= e^{-\langle \lambda,{\rm r}\rangle}
	\int_{M_d(\mathbb{R})} \mathbb{E}[e^{-\langle \lambda,\mathbf{T}_{\rm r+\overline{\textsc{x}}}\rangle} 
	e^{-\llangle \mu,\mathbb{X}_{\mathbf{T}_{\rm r +\overline{\textsc{x}}}\rrangle}}
	\ind_{\{\mathbf{T}_{\rm r+\overline{\textsc{x}}}\in\mathbb{R}^d_+\}}]
	e^{-\llangle\mu,\textsc{x}\rrangle}\mathbb{P}(\mathbb{X}_{\rm r}\in {\rm d}\textsc{x}),\]
where $\textsc{x}=(x^{(1)},\dots,x^{(d)})$ and $\overline{\textsc{x}}=\sum\limits_{j\in[d]} {\rm x}^{(j)}=
\left( \sum\limits_{j\in[d]} x^{1,j},\dots,\sum\limits_{j\in[d]} x^{d,j}\right)$.
This equality can also be written as
\begin{align*}
e^{-\langle r, \Phi(\lambda,\mu)\rangle}
	&= e^{-\langle \lambda,{\rm r}\rangle}
	\int_{M_d(\mathbb{R})} e^{-\langle \rm r+\overline{\textsc{x}},\Phi(\lambda,\mu)\rangle}e^{-\llangle\mu,\textsc{x}\rrangle}
	\mathbb{P}(\mathbb{X}_{\rm r}\in {\rm d}\textsc{x}) 
	\\&= e^{-\langle \lambda,{\rm r}\rangle}
		e^{-\langle r, \Phi(\lambda,\mu)\rangle}
	\mathbb{E}[ e^{-\llangle \mu+\hat{\Phi}(\lambda,\mu),\mathbb{X}_{\rm r}\rrangle}],
\end{align*}
where $\hat{\Phi}(\lambda,\mu)$ is the matrix whose all columns are equal to $\Phi(\lambda,\mu)$. Thanks to the independence 
of the ${\rm X}^{(j)}$'s, the latter equality is reduced to
\[e^{\langle \lambda,{\rm r}\rangle}
		=\prod\limits_{j\in[d]}\mathbb{E}[ e^{-\langle \mu^{(j)}+\Phi(\lambda,\mu),{\rm X}^{(j)}_{r_{j}}\rangle}]\,.\ \]
As a consequence, the Laplace exponent $\Phi$ of $(\mathbf{T}_{\rm r},\mathbb{X}_{\mathbf{T}_{\rm r}})$ satisfy
(\ref{2978}).\\

Now recall the definition of the Esscher transform ${\rm X}^{\mu^{(j)},(j)}$ of each ${\rm X}^{(j)}$ given after Proposition \ref{8427}, 
with Laplace exponent 
\[\varphi_{j}^{\mu^{(j)}}(\lambda)=\varphi_{j}(\lambda+\mu^{(j)})-\varphi_{j}(\mu^{(j)}),\;\;\;
\lambda\in\mathbb{R}^{d}_{+},\;\;\;j\in[d]\,.\ \]
From these Esscher transforms we defined, see $(\ref{3489})$, the spaLf $\mathbf{X}^{\mu}$ by
\[\mathbf{X}^{\mu}_{\rm t} = \sum\limits_{j\in[d]} {\rm X}^{\mu^{(j)},(j)}_{t_{j}}, \;\;\; {\rm t}\in\mathbb{R}^{d}_{+}.\] 
Let $D_\mu=\{\lambda\in\mathbb{R}^d_+:\varphi_{j}^{\mu^{(j)}}(\lambda)>0,j\in[d]\}$. Then under assumption $(H)$,
from part 1.~of Theorem \ref{4526} and from the absolute continuity relationship (\ref{3790}) between $\mathbf{X}$ and
$\mathbf{X}^{\mu}$, the set $D_\mu$ is not empty. Moreover, thanks to Theorem \ref{4526}, the Laplace 
exponent $\varphi^{\mu}=(\varphi_{1}^{\mu^{(1)}},\dots,\varphi_{d}^{\mu^{(d)}})$ of $\mathbf{X}^{\mu}$ 
is a diffeomorphism from $D_\mu$, whose inverse $\phi^{\mu}:(0,\infty)^d\rightarrow D_\mu$ is the Laplace exponent of the 
field $\{\mathbf{T}^{\mu}_{\rm r}, {\rm r}\in\mathbb{R}^{d}_{+}\}$, where $ \mathbf{T}^{\mu}_{\rm r}=\inf\{ {\rm t}\geq 0: 
\mathbf{X}^{\mu}_{\rm t}=-{\rm r}\}$.

On the other hand, from (\ref{2978}), $\Phi$ satisfies
\[\varphi_{j}^{\mu^{(j)}}(\Phi(\lambda,\mu))=\lambda_{j}-\varphi_{j}(\mu^{(j)}), \;\;\; j\in[d],\;\;\;(\lambda,\mu)\in\mathcal{M}_\varphi\,.\ \]
Thus the Laplace exponent $\Phi$ of the couple $(\mathbf{T}_{\rm r},\mathbb{X}_{\mathbf{T}_{\rm r}})$ exists and is given 
for all for $(\lambda,\mu)\in\mathcal{M}_\varphi$ such that $\lambda_j>\varphi_{j}(\mu^{(j)})$, $j\in[d]$ by
\begin{equation}
\Phi(\lambda,\mu)=\phi^{\mu}(\lambda_{1}-\varphi_{1}(\mu^{(1)}),\dots,\lambda_{d}-\varphi_{d}(\mu^{(d)})),
\end{equation}
Finally this relation is extended to the whole set $\mathcal{M}_\varphi$
by continuity. 
\end{proof}

\begin{remark}
We emphasize that Theorem $\ref{4126}$ provides an extension of the case $d=1$. More specifically, $(\ref{2978})$
can be compared to relation $(2)$, p.~$191$ in \cite{be}.
\end{remark}

Let us define the set
\begin{eqnarray*}
\widehat{M}_d(\mathbb{R})&=&\{\textsc{x}\in{M}_d(\mathbb{R}):\mbox{$\textsc{x}$ is essentially nonnegative and 
$\textsc{x}\cdot{\bf 1}\le0$}\}
\end{eqnarray*}
endowed with some matrix norm, $\|\cdot\|$ 
and  equipped with its Borel $\sigma$-field.
From Theorem \ref{4126}, the measure $\p(\mathbf{T}_{\rm r}\in {\rm dt},\,\mathbb{X}_{\rm t}\in {\rm d}\textsc{x}){\rm d}{\rm r}$ 
on $\mathbb{R}_+^d\times\mathbb{R}_+^d\times\widehat{M}_d(\mathbb{R})$ has Laplace transform 
\begin{eqnarray}
&&\int_{\mathbb{R}_+^d\times\mathbb{R}_+^d\times\widehat{M}_d(\mathbb{R})}e^{-\langle\alpha,{\rm r}
\rangle-\langle\lambda,{\rm t}\rangle -\llangle \mu,\textsc{x}\rrangle}\,\p(\mathbf{T}_{\rm r}\in {\rm dt},\,
\mathbb{X}_{\rm t}\in {\rm d}\textsc{x}){\rm d}{\rm r}\nonumber\\
&&\qquad\qquad=
[(\alpha_1+\Phi_{1}(\lambda,\mu))(\alpha_2+\Phi_{2}(\lambda,\mu))\dots(\alpha_d+\Phi_{d}(\lambda,\mu))]^{-1}.\label{6252}
\end{eqnarray}
The following result shows that this measure can be expressed only in terms of the law of the spaLf. 

\begin{theorem}\label{3492} 
Assume that $(H)$ is satisfied. Then for all $\alpha\in\mathbb{R}_+^d$ and $(\lambda,\mu)\in\mathcal{M}_\varphi$,
\begin{eqnarray}
&&\int_{\mathbb{R}_+^d\times\mathbb{R}_+^d\times\widehat{M}_d(\mathbb{R})}e^{-\langle\alpha,{\rm r}
\rangle-\langle\lambda,{\rm t}\rangle -\llangle \mu,\textsc{x}\rrangle}\,\p(\mathbf{T}_{\rm r}\in {\rm dt},\,
\mathbb{X}_{\rm t}\in {\rm d}\textsc{x}){\rm d}{\rm r}\nonumber\\
&&\qquad\qquad=\int_{\mathbb{R}_+^d\times\widehat{M}_d(\mathbb{R})}e^{\langle\alpha,\textsc{x}\cdot{\bf 1}
\rangle-\langle\lambda,{\rm t}\rangle -\llangle \mu,\textsc{x}\rrangle}
\frac{\mbox{det}(-\textsc{x})}{t_1t_2\dots t_d}\p(\mathbb{X}_{\rm t}\in {\rm d}\textsc{x})\,{\rm dt}\,.\label{7370}
\end{eqnarray}
In other terms, the measure 
\[\p(\mathbf{T}_{\rm r}\in {\rm dt},\,
\mathbb{X}_{\rm t}\in {\rm d}\textsc{x}){\rm d}{\rm r},\;\;\;{\rm t}\in\mathbb{R}^d_+,\;{\rm x}\in\widehat{M}_d(\mathbb{R}),\;
{\rm r}\in\mathbb{R}^d_+,\]
is the image of the  measure 
\[\frac{\mbox{det}(-\textsc{x})}{t_1t_2\dots t_d}\p(\mathbb{X}_{\rm t}\in {\rm d}\textsc{x})\,{\rm dt},\;\;\;
{\rm t}\in\mathbb{R}^d_+,\;{\rm x}\in\widehat{M}_d(\mathbb{R}),\]
through the mapping $(t,\textsc{x})\mapsto (t,\textsc{x},-\textsc{x}\cdot{\bf 1})$. 
\end{theorem}
\noindent When $d=1$, the above identity can be read as 
\begin{equation}\label{2202}
\p(T_x\in{\rm d}t){\rm d}x=\frac{-x}{t}\p(X_t\in {\rm d}x)\,{\rm d}t,\;\;\;(t,x)\in(0,\infty)\times(-\infty,0),
\end{equation}
and is known as Kemperman's identity for spectrally positive L\'evy processes. It can be found in \cite{be}, see Proposition VII.2.\\

We shall prove Theorem \ref{3492} through discrete approximation. As a first step, we need to recall the discrete time and space 
counterpart of spaLf's.  Those are matrix valued fields of the form $\{\mathbb{S}_{\rm n},\,{\rm n}\in\mathbb{Z}_+^d\}=
\{(S^{i,j}_{n_j})_{i,j\in[d]},\,{\rm n}\in\mathbb{Z}_+^d\}$, where the columns ${\rm S}^{(j)}={}^t(S^{1,j},\dots,S^{d,j})$, $j\in[d]$ are 
independent random walks. Moreover, all coordinates $S^{i,j}$ start from 0 and take their values in $k^{-1}\mathbb{Z}$, 
where $k\ge1$ is some integer which will be fixed until mentioned otherwise. For $i\neq j$ they are non 
decreasing and for $i=j$ they are downward skip free, that is $S^{i,i}_n-S^{i,i}_{n-1}\ge -k^{-1}$, for all $n\ge1$. 
This setting is introduced in \cite{cl} (for $k=1$ and up to transposition of the matrix $\mathbb{S}$). 
Equivalently to the  continuous case, we define the field $\mathbf{S}:=\mathbb{S}\cdot{\bf 1}$ and its first hitting time process
\[\mathbf{T}_{\rm r}^{\mathbf{S}}=\inf\{{\rm n}:\mathbf{S}_{\rm n}=-{\rm r}\},\;\;\;{\rm r}\in k^{-1}\mathbb{Z}_+^d,\] 
see Lemma 2.2 in \cite{cl}. The field $\mathbb{S}$ (or equivalently $\mathbf{S}$) will be called a downward skip free random 
field (dsfrf for short). An essential result for the proof of the theorem $4.2$, is the following extension of the ballot theorem 
\begin{equation}\label{7377}
\p(\mathbf{T}_{\rm r}^{\mathbf{S}}={\rm n},\mathbf{S}_{\rm n}=\textsc{x})=\frac{k^d\mbox{det}(-\textsc{x})}{n_1\dots n_d}
\p(\mathbf{S}_{\rm n}=\textsc{x}),
\end{equation}
for all ${\rm n}\in\mathbb{N}^d$ and all essentially nonnegative matrix ${\rm x}$ of $M_d(k^{-1}\mathbb{Z})$ such that 
$\textsc{x}\cdot{\bf 1}=-{\rm r}$. (Here we have used the notation $\mathbb{N}=\mathbb{Z}_+\setminus\{0\}$.) Identity (\ref{7377}) is 
proved for $k=1$ in \cite{cl}, see Theorem 3.4 therein. Its extension to any $k\ge1$ is straightforward.

The next step is to consider lattice valued spaLf's. Let us first define these processes. 
Let  ${\rm X}^{(j)}=$ $^t(X^{1,j},\dots,X^{d,j})$, $j\in[d]$ be a family of $d$ independent $d$-dimensional L\'evy processes
such that for $i\neq j$, $X^{i,j}$ is non-decreasing $k^{-1}\mathbb{Z}$-valued L\'evy process and for each $j\in[d]$, $X^{j,j}$ is a $k^{-1}\mathbb{Z}$-valued L\'evy process 
such that for all $t>0$, $X^{j,j}_t-X^{j,j}_{t-}\ge -k^{-1}$. Then there exists a dsfrf $\mathbb{S}$ as defined above and $d$ 
independent Poisson processes $N^{(j)}$, $j\in[d]$ also independent of $\mathbb{S}$ such that 
\begin{equation}\label{3613}
X^{i,j}_t=S^{i,j}_{N^{(j)}_t},\;\;\;i,j\in[d],\;\;\;t\ge0.
\end{equation}
The random fields $\{\mathbb{X}_{\rm t},\,{\rm t}\in\mathbb{R}_+^d\}=\{(X^{i,j}_{t_j})_{i,j\in[d]},\,{\rm t}\in\mathbb{R}_+^d\}$
and $\mathbf{X}=\mathbb{X}\cdot{\bf 1}$ will be referred to as lattice valued spaLf's.
Let $(e_n^{(j)})_{n\ge0}$, $j\in[d]$ be the sequences of exponentially distributed random variables satisfying
\[N^{(j)}_t=\sum_{n\ge0}\ind_{\{e^{(j)}_1+\dots +e^{(j)}_n\le t\}}.\]
The first hitting time process of $\mathbf{X}$ can be defined in the same way as for spaLf's in Lemma \ref{1377} and 
Proposition \ref{8427}. It is denoted by  
\[\mathbf{T}_{\rm r}=\inf\{{\rm t}:\mathbf{X}_{\rm t}=-{\rm r}\},\;\;\;{\rm r}\in k^{-1}\mathbb{Z}_+^d.\] 
We can easily check that the latter is related to the first hitting time process of $\mathbf{S}$ through the identity,
\begin{equation}\label{3673}
T_{\rm r}^{(j)}=\sum_{l=1}^{T_{\rm r}^{(j),\mathbf{S}}}e^{(j)}_l, \;\; j\in[d].
\end{equation}
The following proposition is a direct consequence of (\ref{7377}). Although it can also be found in \cite{ch} for $k=1$, we give 
a more direct proof here. 

\begin{proposition}\label{1636}
Let $\{\mathbb{X}_{\rm t},\,{\rm t}\in\mathbb{R}_+^d\}=\{(X^{i,j}_{t_j})_{i,j\in[d]},\,{\rm t}\in\mathbb{R}_+^d\}$ be a lattice valued 
spaLf. Then for fixed ${\rm r}\in k^{-1}\mathbb{Z}_+^d$, the joint law of 
$(\mathbf{T}_{\rm r},\mathbb{X}_{\mathbf{T}_{\rm r}})$ is given by 
\[\p(\mathbf{T}_{\rm r}\in {\rm dt},\,\mathbb{X}_{\rm t}=\textsc{x})
=\frac{k^d\mbox{det}(-\textsc{x})}{t_1t_2\dots t_d}\p(\mathbb{X}_{\rm t}=\textsc{x}){\rm d}t_1{\rm d}t_2\dots {\rm d}t_d\,,\]
for all essentially nonnegative matrices $\textsc{x}$ of $M_d(k^{-1}\mathbb{Z})$ such that $\textsc{x}\cdot{\bf 1}=-{\rm r}$.
\end{proposition}
\begin{proof}
Let ${\rm r}$ and $\textsc{x}=(x_{i,j})_{i,j\in[d]}$ be as in the statement. Then the straightforward identity 
$\mathbb{S}_{\mathbf{T}_{\rm r}^{\mathbf{S}}}=\mathbb{X}_{\mathbf{T}_{\rm r}}$ together with expressions (\ref{3613}) and 
(\ref{3673}) allow us to write,
\begin{align*}
\mathbb{P}(\mathbf{T}_{r}\in {\rm dt},\mathbb{X}_{\rm t}=\textsc{x})
        &=\mathbb{P}\left(\sum\limits_{l=1}^{T_{\rm r}^{(j),\mathbf{S}}} e_{l}^{(j)}\in {\rm d}t_{j},\,j\in[d],\,
	\mathbb{S}_{\mathbf{T}_{\rm r}^{\mathbf{S}}}=\textsc{x}\right)
	\\&= \sum\limits_{{\rm n}\in\mathbb{N}^{d}}\prod_{j\in[d]}
	\mathbb{P}\left(\sum\limits_{l=1}^{n_{j}} e_{l}^{(j)}\in {\rm d}t_{j}\right)
	\mathbb{P}(\mathbf{T}_{\rm r}^{\mathbf{S}}={\rm n}, \mathbb{S}_{\rm n}=\textsc{x})
	\\&=\sum\limits_{{\rm n}\in\mathbb{N}^{d}}
	\dfrac{\lambda_{1}^{n_{1}}t_{1}^{n_{1}}\dots\lambda_{d}^{n_d}t_{d}^{n_{d}}}{n_{1}!\dots n_{d}!}e^{-\langle 
	{\rm \lambda},{\rm t}\rangle}
	\dfrac{k^{-d}\det(-\textsc{x})}{t_{1}\dots t_{d}}\mathbb{P}(\mathbb{S}_{\rm n}=\textsc{x}) {\rm dt}
	\\&= \dfrac{k^{-d}\det(-\textsc{x})}{t_{1}\dots t_{d}}\sum\limits_{{\rm n}\in\mathbb{N}^{d}}
	\prod_{j\in[d]}\p(N_{t_j}^{(j)}=n_j)\mathbb{P}(\mathbb{S}_{\rm n}=\textsc{x}) {\rm dt}
	\\&=\dfrac{k^{-d}\det(-\textsc{x})}{t_{1}\dots t_{d}}\mathbb{P}(\mathbb{X}_{\rm t}=\textsc{x}){\rm dt},
\end{align*}
which proves our result.
\end{proof}

From now on, we will add $k$ as a superscript to all objects referring to the discrete valued spaLf defined 
above. For instance, the latter will be denoted by $\mathbb{X}^{(k)}=(X^{i,j,k})_{i,j\in[d]}$ or $\mathbf{X}^{(k)}$,
where ${\rm X}^{(j),k}= {}^{t}(X^{1,j,k},\dots,X^{d,j,k})$. It is pretty clear that lattice valued spaLf's satisfy analogous properties 
to those of spaLf's introduced in Section \ref{2018}. In particular, the discrete time field 
${\rm r}\mapsto(\mathbf{T}_{\rm r}^{(k)},\mathbb{X}_{\mathbf{T}_{\rm r}^{(k)}}^{(k)})$, ${\rm r}\in k^{-1}\mathbb{Z}^d_+$ 
has independent and stationary increments and can be treated in a very similar way as its continuous space counterpart 
involved in Theorem \ref{4126}. That is why we will content ourselves with stating the next theorem as well as some 
preliminary results without giving any proof.\\

Recall the definition of the Laplace exponent $\varphi_j^{(k)}$ of ${\rm X}^{(j),k}$, that is
\[\e[e^{-\langle \lambda, {\rm X}^{(j),k}_t\rangle}]=e^{t\varphi_j^{(k)}(\lambda)}\,,\;\;\;t\ge0\,,\;\;\;{\bf\lambda}=
(\lambda_1,\dots,\lambda_d)\in \mathbb{R}_+^d\,.\] 
Then as in Theorem \ref{4526}, we can prove that the hypothesis
\[(H^{(k)})\;\;\;\;\;\mbox{$D^{(k)}:=\{\lambda\in\mathbb{R}_+^d:\varphi_j^{(k)}(\lambda)>0,\,j\in[d]\}$ is non empty}\]
is equivalent to the fact that $\mathbf{T}_{\rm r}^{(k)}\in\mathbb{R}_+^d$ holds with positive probability, for all 
${\rm r}\in k^{-1}\mathbb{Z}_+^d$. As in Theorem \ref{4526}, the proof of this equivalence is based on the Esscher transform 
$\mathbf{X}^{(k),\mu}$, for $\mu\in M_d(\mathbb{R}_+)$ whose Laplace exponent is given by 
\begin{equation}\label{9420}
\varphi_j^{(k),\mu^{(j)}}(\lambda)=\varphi_j^{(k)}(\lambda+\mu^{(j)})-\varphi_j^{(k)}(\mu^{(j)}),\;\;\;\;\lambda\in\mathbb{R}_+^d.
\end{equation}
Let us define the set 
\[\mathcal{M}_\varphi^{(k)}=\{(\lambda,\mu)\in\mathbb{R}^{d}_{+}\times{M}_d(\mathbb{R}_{+}):
\lambda_{j}\ge\varphi_{j}^{(k)}(\mu^{(j)}),\,j\in[d]\}.\]
The following theorem is the analog of Theorem \ref{4126} for lattice valued spaLf's. 

\begin{theorem}\label{4226} Assume that $(H^{(k)})$ holds. Let ${\rm r}=(r_1,\dots,r_d)\in k^{-1}\mathbb{Z}_+^d$ and let 
$\mathbf{T}_{\rm r}^{(k)}$ be the first hitting time of level $-{\rm r}$ by the spaLf $\mathbf{X}^{(k)}$, then there exits a mapping 
$\Phi^{(k)} : \mathcal{M}_\varphi^{(k)} \rightarrow \mathbb{R}^{d}_{+}$ such that
\[\mathbb{E}\left[e^{-\langle\lambda,\mathbf{T}_{\rm r}^{(k)}\rangle -\llangle \mu,
\mathbb{X}^{(k)}_{\mathbf{T}_{\rm r}^{(k)}}\rrangle}\ind_{\{\mathbf{T}^{(k)}_{\rm r}\in\mathbb{R}^d_+\}}\right]=
e^{-\langle {\rm r},\Phi^{(k)}(\lambda,\mu)\rangle},\;\;\; (\lambda,\mu)\in\mathcal{M}^{(k)}_\varphi.\]
Moreover $\Phi^{(k)}$ satisfies the equations, 
\begin{equation}\label{2979}
\varphi_{j}^{(k)}(\mu^{(j)}+
\Phi^{(k)}(\lambda,\mu))=\lambda_{j}, \;\;\; j\in[d],\;\;\;(\lambda,\mu)\in\mathcal{M}^{(k)}_\varphi,
\end{equation}
and it is explicitly determined by
\begin{equation}\label{8779}
\Phi^{(k)}(\lambda,\mu)=\phi^{(k),\mu}(\lambda_{1}-\varphi^{(k)}_{1}(\mu^{(1)}),\dots,\lambda_{d}-\varphi^{(k)}_{d}(\mu^{(d)})),
\end{equation}
where $\phi^{(k),\mu}$ is the inverse of the Laplace exponent $\varphi^{(k),\mu}$ of the Esscher transform $\mathbf{X}^{(k),\mu}$ 
recalled in $(\ref{9420})$. 
\end{theorem}
 
In order to end the proof of Theorem \ref{4126}, we need to prove that any $d$-dimensional L\'evy process is the weak limit of a 
sequence of lattice valued $d$-dimensional L\'evy processes.  The index $k$ is now a variable that will be taken  to infinity. 

\begin{lemma}\label{2895}
Let ${\rm Y}$ be any $d$-dimensional L\'evy process. Then there exists a sequence of $(k^{-1}\mathbb{Z})^d$-valued  
L\'evy processes ${\rm Y}^{(k)}$ which converges weakly in the $J_1$ Skohorod's topology toward ${\rm Y}$.
\end{lemma}
\begin{proof}  Let us first assume that $Y$ has bounded variation. Then the characteristic exponent $\psi$ of $Y$ 
can be written as 
\[\psi(\lambda)=-i\langle {\rm a},\lambda\rangle-\int_{\mathbb{R}^d}(1-e^{i\langle \lambda,{\rm x}\rangle})\,\pi(d{\rm x}),\;\; 
\lambda\in\mathbb{R}^{d},\]
where ${\rm a}=(a_1,\dots,a_d)\in\mathbb{R}^d$ and $\pi$ is some L\'evy measure such that 
$\int_{\mathbb{R}^d}(1\wedge|{\rm x}|)\,\pi(d{\rm x})<\infty$.

Let $\pi^{(k)}$ be the restriction of $\pi$ to the set 
$\{{\rm x}\in\mathbb{R}^d:|{\rm x}|\ge k^{-1}\}$ i.e.~$\pi^{(k)}(d{\rm x})=\ind_{\{{\rm x}\in\mathbb{R}^d:|{\rm x}|\ge k^{-1}\}}\pi(d{\rm x})$. 
For $x\in\mathbb{R}$, set $\mbox{sign}(x)=\ind_{\{x>0\}}-\ind_{\{x<0\}}$. Then we consider the following sequence of 
$(k^{-1}\mathbb{Z})^d$-valued L\'evy processes
\[{\rm Y}^{(k)}_t=k^{-1}\tilde{{\rm N}}^{(k)}_{t}+\sum_{n=0}^{N_t^{(k)}}{\rm Z}_n^{(k)},\]
where $\tilde{{\rm N}}^{(k)}=(\mbox{sign}(a_1)\tilde{N}^{1,k},\dots,\mbox{sign}(a_d)\tilde{N}^{d,k})$ and  
$\tilde{N}^{1,k},\dots,\tilde{N}^{d,k}$ are independent Poisson processes with respective intensities $k|a_j|$,  
$(N_t^{(k)})_{t\ge0}$ is a Poisson process with intensity  $\pi(k ^{-1},\infty)^d$ and for each $k\ge1$, $({\rm Z}_n^{(k)})_{n\ge0}$ 
is a sequence of i.i.d~random variables such that ${\rm Z}_n ^{(k)}\ed k^{-1}[k{\rm Z}_k]$  and ${\rm Z}_k$ has  law 
$(\pi(k^{-1},\infty)^d)^{-1}\pi^{(k)}(d{\rm x})$. (Here $[{\rm x}]=([x_1],\dots,[x_d])$ and $[x_i]$ denotes the lower integer part of 
$x_{i}\in\mathbb{R}$.) Moreover, the sequences $\{(\tilde{{\rm N}}_t^{(k)})_{t\ge0}, k\ge1\}$, $\{(N_t^{(k)})_{t\ge0}, k\ge1\}$ 
and $\{({\rm Z}_n^{(k)})_{n\ge0},k\ge1\}$ are independent. Then we can check that ${\rm Y}^{(k)}$ has characteristic exponent 
\[\psi_{k}(\lambda)=\sum_{j=1}^dk|a_{j}|\left(1-e^{i\frac{\lambda_j\mbox{\tiny sign}(a_j)}k}\right)+
\int_{(0,\infty)^d}(1-e^{i\langle \lambda,{\rm x}\rangle})\,\pi^{(k)}(d{\rm x}), \;\; \lambda\in\mathbb{R}^{d}_{+},\]
whose limit, as $k$ tends to $\infty$, is $\psi(\lambda)$, for all $\lambda\in\mathbb{R}^{d}$. It proves that the sequence of random 
variables $({\rm Y}_1^{(k)})_{k\ge1}$ converges weakly towards ${\rm Y}_1$.

Then recall that from Theorem 2.7 in \cite{sk}, which can be extended in higher dimension, see Section 5 in the same paper, 
the weak convergence of the sequence of random variables $({\rm Y}_1^{(k)})_{k\ge1}$ toward ${\rm Y}_1$ implies the weak 
convergence of the sequence of processes $\{({\rm Y}^{(k)}_t)_{t\ge0}, \;k\ge1\}$ towards $({\rm Y}_t)_{t\ge0}$ in the $J_1$ 
Skohorod's topology.  

Let us now assume that ${\rm Y}$ is any L\'evy process. From Lemma 45.12 in \cite{sa}, there is a sequence of compound
Poisson processes ${\rm Z}^{(k)}$ which converges weakly in the $J_1$ Skohorod's topology (and even in the sense of the 
uniform convergence) toward ${\rm Y}$. In application of what has just been proved, for each $k$, there is a sequence of 
$(n^{-1}\mathbb{Z})^d$-valued L\'evy processes $({\rm Z}^{(n,k)})_{n\geq 1}$ which converges weakly in the $J_1$ Skohorod's 
topology toward ${\rm Z}^{(k)}$. Then it suffices to set  ${\rm Y}^{(k)}:={\rm Z}^{(k,k)}$ in order to obtain the desired sequence.
\end{proof}

We have now gathered all necessary ingredients for the proof of Theorem \ref{3492}.\\

\noindent {\it Proof of Theorem $\ref{3492}$}. Let $(\mathbb{X}^{(k)})_{k\ge1}$ be a sequence of lattice valued spaLf's 
such that each sequence of columns $({\rm X}^{(j),k})_{k\geq 1}$, where ${\rm X}^{(j),k}= {}^{t}(X^{1,j,k},\dots,X^{d,j,k})$, 
converges weakly to ${\rm X}^{(j)}$. 
The existence of such a sequence is ensured by lemma \ref{1636}. This convergence means in particular that 
\begin{equation}\label{2823}
\lim_{k\rightarrow\infty}\varphi_j^{(k)}(\lambda)=\varphi_j(\lambda)\,,\;\;\;\;\lambda\ge0,\;\;\;j\in[d]\,.
\end{equation}
Since $(H)$ is satisfied, by continuity of the functions $\varphi_j$ and from (\ref{2823}), there is $k_0$ such that for all
$k\ge k_0$, $(H^{(k)})$ is satisfied. Then let $k\ge k_0$ and let $\widehat{M}_{d,{\rm r}}(k^{-1}\mathbb{Z})$ be the set of 
essentially nonnegative matrices $\textsc{x}$ of $M_d(k^{-1}\mathbb{Z})$ such that $\textsc{x}\cdot{\bf 1}=-{\rm r}$.  
We derive from Theorem \ref{4226} that for all $\alpha\in\mathbb{R}_+^d$ and $(\lambda,\mu)\in\mathcal{M}_\varphi^{(k)}$,
\begin{eqnarray*}\label{7322}
&&\qquad\sum_{{\rm r}\in k^{-1}\mathbb{Z}_+^d}k^{-d}e^{-\langle\alpha,{\rm r}\rangle}\int_{\mathbb{R}_+^d}\sum_{{\rm x}\in 
\widehat{M}_{d,{\rm r}}(k^{-1}\mathbb{Z})}e^{-\langle\lambda,{\rm t}\rangle -\llangle \mu,\textsc{x}\rrangle}\,
\p(\mathbf{T}^{(k)}_{\rm r}\in {\rm dt},\,\mathbb{X}^{(k)}_{\rm t}=\textsc{x})\\ 
&&=
[k(1-e^{-k^{-1}(\alpha_1+\Phi^{(k)}_{1}(\lambda,\mu))})\times k(1-e^{-k^{-1}(\alpha_2+\Phi^{(k)}_{2}(\lambda,\mu))})\dots
\times k(1-e^{-k^{-1}(\alpha_d+\Phi^{(k)}_{d}(\lambda,\mu))})]^{-1}.\nonumber
\end{eqnarray*}
Now take $(\lambda,\mu)\in\mathcal{M}_\varphi$ such that $\lambda_{j}>\varphi_{j}(\mu^{(j)})$ for all $j\in[d]$. Then 
by continuity of $\varphi_j$, $j\in[d]$, there is $k'_0$ such that for all $k\ge k'_0$, $(\lambda,\mu)\in\mathcal{M}_\varphi^{(k)}$.  
Clearly $(\varphi_j^{(k),\mu^{(j)}})_{k\geq 1}$ defined in (\ref{9420}) converges pointwise to $\varphi_j^{\mu^{(j)}}$, for all $j\in[d]$.
Hence, the sequence of inverses $(\phi^{(k),\mu})_{k\geq 1}$ also converges pointwise to $\phi^{\mu}$. Therefore, from 
(\ref{8778}), (\ref{8779}) and by continuity, $(\Phi^{(k)}(\lambda,\mu))_{k\geq 1}$ converges to $\Phi(\lambda,\mu)$. 

Now let us extend the definition of $\mathbf{T}^{(k)}_{\rm r}$ to all ${\rm r}\in\mathbb{R}_+^d$ by setting 
$\mathbf{T}^{(k)}_{\rm r}:=\mathbf{T}^{(k)}_{[\rm r]}$, where $[{\rm r}]=k^{-1}([kr_1],\dots,[kr_d])$ and where $[x]$ denotes the 
lower integer part of $x$. Then by taking $k$ to infinity in (\ref{7322}), we obtain from (\ref{6252}) that for all 
$\alpha\in\mathbb{R}^d_+$ and $(\lambda,\mu)\in\mathcal{M}_\varphi$ such that $\lambda_{j}>\varphi_{j}(\mu^{(j)})$, 
for all $j\in[d]$,
\begin{eqnarray}
&&\lim_{k\rightarrow\infty}\sum_{{\rm r}\in k^{-1}\mathbb{Z}_+^d}k^{-d}e^{-\langle\alpha,{\rm r}\rangle}
\int_{\mathbb{R}_+^d}\sum_{{\rm x}\in\widehat{M}_{d,{\rm r}}(k^{-1}\mathbb{Z})}e^{-\langle\lambda,{\rm t}\rangle 
-\llangle \mu,\textsc{x}\rrangle}\,\p(\mathbf{T}^{(k)}_{\rm r}\in {\rm dt},\,\mathbb{X}^{(k)}_{\rm t}=\textsc{x})\label{5951}\\ 
&=&\lim_{k\rightarrow\infty}\int_{\mathbb{R}_+^d\times\mathbb{R}_+^d\times\widehat{M}_d(\mathbb{R})}e^{-\langle\alpha,{\rm r}
\rangle-\langle\lambda,{\rm t}\rangle -\llangle \mu,\textsc{x}\rrangle}\,\p(\mathbf{T}^{(k)}_{\rm r}\in {\rm dt},\,
\mathbb{X}^{(k)}_{\rm t}\in {\rm d}\textsc{x}){\rm d}{\rm r}\nonumber\\
&=&[(\alpha_1+\Phi_{1}(\lambda,\mu)))(\alpha_2+\Phi_{2}(\lambda,\mu)))\dots(\alpha_d+\Phi_{d}(\lambda,\mu))]^{-1}\nonumber\\
&=&\int_{\mathbb{R}_+^d\times\mathbb{R}_+^d\times\widehat{M}_d(\mathbb{R})}e^{-\langle\alpha,{\rm r}
\rangle-\langle\lambda,{\rm t}\rangle -\llangle \mu,\textsc{x}\rrangle}\,\p(\mathbf{T}_{\rm r}\in {\rm dt},\,
\mathbb{X}_{\rm t}\in {\rm d}\textsc{x}){\rm d}{\rm r}.\nonumber
\end{eqnarray}

On the other hand, let $\widehat{M}_{d}(k^{-1}\mathbb{Z})$ be the set of essentially nonnegative matrices $\textsc{x}$ of 
$M_d(k^{-1}\mathbb{Z})$ such that $\textsc{x}\cdot{\bf 1}\le0$. Then as a direct consequence of Proposition \ref{1636}, 
we obtain that for all $\alpha\in\mathbb{R}_+^d$ and $(\lambda,\mu)\in\mathcal{M}_\varphi^{(k)}$,
\begin{eqnarray*}
&&\sum_{{\rm r}\in k^{-1}\mathbb{Z}_+^d}k^{-d}e^{-\langle\alpha,{\rm r}\rangle}\int_{\mathbb{R}_+^d}\sum_{{\rm x}\in 
\widehat{M}_{d,{\rm r}}(k^{-1}\mathbb{Z})}e^{-\langle\lambda,{\rm t}\rangle -\llangle \mu,\textsc{x}\rrangle}\,
\p(\mathbf{T}^{(k)}_{\rm r}\in {\rm dt},\,\mathbb{X}^{(k)}_{\rm t}=\textsc{x})\nonumber\\
&&\qquad\qquad=\int_{\mathbb{R}_+^d}\sum_{{\rm r}\in k^{-1}\mathbb{Z}_+^d,\,{\rm x}\in\widehat{M}_{d,{\rm r}}(k^{-1}
\mathbb{Z})}e^{-\langle\alpha,{\rm r}\rangle-\langle\lambda,{\rm t}\rangle -\llangle \mu,\textsc{x}\rrangle}
\frac{\mbox{det}(-\textsc{x})}{t_1t_2\dots t_d}\p(\mathbb{X}^{(k)}_{\rm t}=\textsc{x})\,{\rm dt}\,.\nonumber\\
&&\qquad\qquad=\int_{\mathbb{R}_+^d}\sum_{{\rm x}\in \widehat{M}_{d}(k^{-1}\mathbb{Z})}
e^{\langle\alpha,\textsc{x}\cdot{\bf 1}\rangle-\langle\lambda,{\rm t}\rangle -\llangle \mu,\textsc{x}\rrangle}
\frac{\mbox{det}(-\textsc{x})}{t_1t_2\dots t_d}\p(\mathbb{X}^{(k)}_{\rm t}=\textsc{x})\,{\rm dt}\nonumber\\
&&\qquad\qquad=\int_{\mathbb{R}_+^d}\int_{\widehat{M}_{d}(\mathbb{R})}
e^{\langle\alpha,\textsc{x}\cdot{\bf 1}\rangle-\langle\lambda,{\rm t}\rangle -\llangle \mu,\textsc{x}\rrangle}
\frac{\mbox{det}(-\textsc{x})}{t_1t_2\dots t_d}\p(\mathbb{X}^{(k)}_{\rm t}\in {\rm d}\textsc{x})\,{\rm dt}\,,\label{6214}
\end{eqnarray*}
then it follows from the above calculation and from (\ref{5951}) that for all $\alpha\in\mathbb{R}^d_+$ and 
$(\lambda,\mu)\in\mathcal{M}_\varphi$ such that 
$\lambda_{j}>\varphi_{j}(\mu^{(j)})$, $j\in[d]$,
\begin{eqnarray}
&&\int_{\mathbb{R}_+^d\times\mathbb{R}_+^d\times\widehat{M}_d(\mathbb{R})}e^{-\langle\alpha,{\rm r}
\rangle-\langle\lambda,{\rm t}\rangle -\llangle \mu,\textsc{x}\rrangle}\,\p(\mathbf{T}_{\rm r}\in {\rm dt},\,
\mathbb{X}_{\rm t}\in {\rm d}\textsc{x}){\rm d}{\rm r}\nonumber\\
&&\qquad\quad=\lim_{k\rightarrow\infty}\int_{\mathbb{R}_+^d}\int_{\widehat{M}_{d}(\mathbb{R})}
e^{\langle\alpha,\textsc{x}\cdot{\bf 1}\rangle-\langle\lambda,{\rm t}\rangle -\llangle \mu,\textsc{x}\rrangle}
\frac{\mbox{det}(-\textsc{x})}{t_1t_2\dots t_d}\p(\mathbb{X}^{(k)}_{\rm t}\in {\rm d}\textsc{x})\,{\rm dt}\,.\label{6014}
\end{eqnarray}
Now, we derive from the weak convergence of $\mathbb{X}^{(k)}_{\rm t}$ toward $\mathbb{X}_{\rm t}$ for each ${\rm t}$ that  
\[\lim_{k\rightarrow\infty}\int_{\widehat{M}_{d}(\mathbb{R})}e^{\langle\alpha,\textsc{x}\cdot{\bf 1}\rangle-\llangle 
\mu,\textsc{x}\rrangle}\mbox{det}(-\textsc{x})\p(\mathbb{X}^{(k)}_{\rm t}\in {\rm d}\textsc{x})=
\int_{\widehat{M}_{d}(\mathbb{R})}e^{\langle\alpha,\textsc{x}\cdot{\bf 1}\rangle-\llangle 
\mu,\textsc{x}\rrangle}\mbox{det}(-\textsc{x})\p(\mathbb{X}_{\rm t}\in {\rm d}\textsc{x}),\]
so that for all $\varepsilon>0$, 
\begin{eqnarray*}
&&\lim_{k\rightarrow\infty}\int_{\{{\rm t}\ge\varepsilon\cdot{\bf 1}\}}\int_{\widehat{M}_{d}(\mathbb{R})}
e^{\langle\alpha,\textsc{x}\cdot{\bf 1}\rangle-\langle\lambda,{\rm t}\rangle -\llangle \mu,\textsc{x}\rrangle}
\frac{\mbox{det}(-\textsc{x})}{t_1t_2\dots t_d}\p(\mathbb{X}^{(k)}_{\rm t}\in {\rm d}\textsc{x})\,{\rm dt}\\
&&\qquad\qquad=
\int_{\{{\rm t}\ge\varepsilon\cdot{\bf 1}\}}\int_{\widehat{M}_{d}(\mathbb{R})}
e^{\langle\alpha,\textsc{x}\cdot{\bf 1}\rangle-\langle\lambda,{\rm t}\rangle -\llangle \mu,\textsc{x}\rrangle}
\frac{\mbox{det}(-\textsc{x})}{t_1t_2\dots t_d}\p(\mathbb{X}_{\rm t}\in {\rm d}\textsc{x})\,{\rm dt}.
\end{eqnarray*}
Then from Proposition \ref{1636},
\begin{eqnarray*}
&&\int_{\{{\rm t}\ge\varepsilon\cdot{\bf 1}\}^c}\int_{\widehat{M}_{d}(\mathbb{R})}
e^{\langle\alpha,\textsc{x}\cdot{\bf 1}\rangle-\langle\lambda,{\rm t}\rangle -\llangle \mu,\textsc{x}\rrangle}
\frac{\mbox{det}(-\textsc{x})}{t_1t_2\dots t_d}\p(\mathbb{X}^{(k)}_{\rm t}\in {\rm d}\textsc{x})\,{\rm dt}\\
&&\qquad\qquad=\int_{\{{\rm t}\ge\varepsilon\cdot{\bf 1}\}^c\times\mathbb{R}_+^d\times\widehat{M}_{d}(\mathbb{R})}
e^{-\langle\alpha,{\rm r}\rangle -\langle\lambda,{\rm t}\rangle -\llangle \mu,\textsc{x}\rrangle}\,\p(\mathbf{T}^{(k)}_{\rm r}\in 
{\rm dt},\,\mathbb{X}^{(k)}_{\rm t}=\textsc{x})\,{\rm dr}\\
&&\qquad\qquad=\int_{\mathbb{R}_+^d}e^{-\langle\alpha,{\rm r}\rangle}\e\left[e^{-\langle\lambda,
\mathbf{T}^{(k)}_{\rm r}\rangle-\llangle\mu,\mathbb{X}^{(k)}_{\mathbf{T}^{(k)}_{\rm r}}\rrangle}
\mathds{1}_{\{\mathbf{T}^{(k)}_{\rm r}\ge\varepsilon\cdot{\bf 1}\}^c}\right]\,{\rm dr},
\end{eqnarray*}
which entails from a trivial extension of (\ref{6014}) that,
\begin{eqnarray}
&&\lim_{k\rightarrow\infty}\int_{\{{\rm t}\ge\varepsilon\cdot{\bf 1}\}^c}\int_{\widehat{M}_{d}(\mathbb{R})}
e^{\langle\alpha,\textsc{x}\cdot{\bf 1}\rangle-\langle\lambda,{\rm t}\rangle -\llangle \mu,\textsc{x}\rrangle}
\frac{\mbox{det}(-\textsc{x})}{t_1t_2\dots t_d}\p(\mathbb{X}^{(k)}_{\rm t}\in {\rm d}\textsc{x})\,{\rm dt}\nonumber\\
&&\qquad\qquad=\int_{\mathbb{R}_+^d}e^{-\langle\alpha,{\rm r}\rangle}\e[e^{-\langle\lambda,\mathbf{T}_{\rm r}\rangle 
-\llangle\mu,\mathbb{X}_{\mathbf{T}_{\rm r}}\rrangle}\mathds{1}_{\{\mathbf{T}_{\rm r}\ge\varepsilon\cdot{\bf 1}\}^c}]\,{\rm dr}.
\label{2699}
\end{eqnarray}
But from part 2.~of Theorem \ref{4526}, for all $i,j\in[d]$, $\lim\limits_{s\rightarrow\infty}\phi_j(s{\rm e}_i)=\infty$, which implies that 
for all ${\rm r}>0$ and all $i\in[d]$,  $\p(T^{(i)}_{\rm r}>0)>0$. In particular,
\[\lim\limits_{\varepsilon\rightarrow0}\p(\{\mathbf{T}_{\rm r}\ge\varepsilon\cdot{\bf 1}\}^c)\le
\lim\limits_{\varepsilon\rightarrow0}\sum\limits_{i=1}^d\p(T^{(i)}_{\rm r}<\varepsilon)=0, \] therefore  by dominated convergence,
expression (\ref{2699}) can be made arbitrarily small as $\varepsilon$ tends to 0. 

Then we have proved that the identity (\ref{7370}) is valid for all $\alpha\in\mathbb{R}^d_+$ and 
$(\lambda,\mu)\in\mathcal{M}_\varphi$ such that $\lambda_{j}>\varphi_{j}(\mu^{(j)})$, $j\in[d]$. Now let any 
$(\lambda,\mu)\in\mathcal{M}_\varphi$ and assume that 
$\lambda_{i}=\varphi_{i}(\mu^{(i)})$ for some $i\in[d]$. Then identity (\ref{7370}) is valid if we replace $\lambda_i$ by 
$\lambda_{i}'=\lambda_{i}+\varepsilon_i$, for $\varepsilon_i>0$ and we obtain it for $(\lambda,\mu)$ by letting $\varepsilon_i$
going to 0 and applying monotone convergence. $\hfill\Box$\\

\noindent{\it Proof of Corollary $\ref{8466}$.}  Assume first that $d>1$. Then taking $\mu=0$ in Theorem \ref{3492} 
gives
\begin{eqnarray}
\int_{\mathbb{R}_+^d\times\mathbb{R}_+^d}
e^{-\langle\alpha,{\rm r}
\rangle-\langle\lambda,{\rm t}\rangle}\,\p(\mathbf{T}_{\rm r}\in {\rm dt}){\rm d}{\rm r}
&=&\int_{\mathbb{R}_+^d\times\widehat{M}_d(\mathbb{R})}e^{\langle\alpha,\textsc{x}\cdot{\bf 1}
\rangle-\langle\lambda,{\rm t}\rangle}
\frac{\mbox{det}(-\textsc{x})}{t_1t_2\dots t_d}\p(\mathbb{X}_{\rm t}\in {\rm d}\textsc{x})\,{\rm dt}\nonumber\\
&=&\int_{\mathbb{R}_+^d}e^{-\langle\lambda,{\rm t}\rangle}
\mathbb{E}\left[e^{\langle\alpha,\mathbb{X}_{\rm t}\cdot{\bf 1}\rangle}
\frac{\mbox{det}(-\mathbb{X}_{\rm t})}{t_1t_2\dots t_d}\ind_{\left\{\mathbb{X}_{\rm t}\in\widehat{M}_d(\mathbb{R})\right\}}\right]
{\rm dt}\,.
\label{7302}
\end{eqnarray}
Note that from our assumptions the density $p_{\rm t}:M_d(\mathbb{R})\rightarrow \mathbb{R}$ of $\widehat{\mathbb{X}}_{\rm t}$ 
is continuous on the set of matrices whose columns belong to $F_1\times F_2\times\dots\times F_d$.
Let $\overline{M}_d(\mathbb{R})$ be the set of essentially nonnegative matrices whose elements of the diagonal are non-positive. 
Then
\[\mathbb{E}\left[e^{\langle\alpha,\mathbb{X}_{\rm t}\cdot{\bf 1}\rangle}
\frac{\det(-\mathbb{X}_{\rm t})}{t_1t_2\dots t_d}\ind_{\left\{\mathbb{X}_{\rm t}\in\widehat{M}_d(\mathbb{R})\right\}}\right]
	 =\int\limits_{\overline{M}_d(\mathbb{R})}
e^{\sum\limits_{i=1}^{d} \alpha_{i}x_{i,i}}
\frac{\det(-(\overline{\textsc{x}}+D(\textsc{x}))}{t_{1}\dots t_{d}} p_{\rm t}(\textsc{x}) {\rm d}\textsc{x},\]
where $D(\textsc{x})=(d_{i,j})_{i,j\in[d]}$ is defined by $d_{i,i}=x_{i,i}$ and $d_{i,j}=0$ for $i\neq j$, and 
$\overline{x}=(\overline{x}_{i,j})_{i,j\in[d]}$ such that $\overline{x}_{i,i}=-\sum\limits_{j\neq i} x_{i,j}$ and 
$\overline{x}_{i,j}=x_{i,j}$ for $i\neq j$. Let $I_{d}$ be the identity matrix. Then
\begin{equation}\label{7722}
\int\limits_{\overline{M}_d(\mathbb{R})}
e^{\sum\limits_{i=1}^{d} \alpha_{i} x_{i,i}}
\frac{\det(-(\overline{\textsc{x}}+D(\textsc{x}))}{t_{1}\dots t_{d}} p_{\rm t}(\textsc{x}) {\rm d}\textsc{x}
	=\int\limits_{\mathbb{R}^{d}_{+}}
	\int\limits_{\mathbb{R}^{d(d-1)}_{+}}
e^{-\langle \alpha,{\rm r}\rangle}
\frac{\det(-(\overline{\textsc{x}}+{\rm r}I_{d}))}{t_{1}\dots t_{d}} p_{\rm t}(\textsc{x}^{\rm r}) \prod_{k\neq j}{\rm d}x_{k,j}{\rm dr},
\end{equation}
where $\textsc{x}^{\rm r}$ is the matrix $\textsc{x}$ in which the variable $x_{i,i}$ has been replaced by $r_{i}$, for all $i\in[d]$.
Then we derive from (\ref{7302}) and (\ref{7722}) that for fixed ${\rm r}\in\mathbb{R}^{d}_{+}$, 
\begin{equation}\label{exprballot}
\mathbb{P}(\mathbf{T}_{\rm r}\in {\rm dt})
	= \int\limits_{\mathbb{R}^{d(d-1)}_{+}}
\frac{\det(-(\overline{\textsc{x}}+{\rm r}I_{d}))}{t_{1}\dots t_{d}} p_{\rm t}(\textsc{x}^{\rm r}) \prod_{k\neq j}{\rm d}x_{k,j}{\rm dt}\,.\
\end{equation}
Let $i\in[d]$ and ${\rm r}=r{\rm e}_{i}$, then
\[\mathbb{P}(\mathbf{T}_{\rm r}\in {\rm dt})
	= \int\limits_{\mathbb{R}^{d(d-1)}_{+}}
\frac{r\det(-\overline{\textsc{x}}^{i,i})}{t_{1}\dots t_{d}} p_{\rm t}(\textsc{x}^{\rm r}) \prod_{k\neq j}{\rm d}x_{k,j}{\rm dt},\]
where $\overline{\textsc{x}}^{i,i}$ is the matrix obtained from $\overline{\textsc{x}}$ by deleting the row and the column $i$. 
From Exercise $1.$~in Chapter $I$ of \cite{be}, the L\'evy measure of the subordinator
$(\mathbf{T}_{r{\rm e}_{i}})_{r\geq 0}$ is the vague limit of $\mathbb{P}(\mathbf{T}_{\rm r}\in {\rm dt})/r$
as $r$ tends to 0, on sets of the form $\{|{\rm t}|>a\}$, $a>0$. Hence the expression of the statement follows from continuity 
property of $p_{\rm t}$.

The expression for $d=1$ is obtained in the same way by using the simpler form (\ref{2202}) of 
$\mathbb{P}(\mathbf{T}_{\rm r}\in {\rm dt})$ in this case. 
$\hfill\Box$

%\newpage
  

\vspace*{.5in}

\begin{thebibliography}{100}
\bibitem{ar} \sc K.J.~Arrow: \rm A "dynamic'' proof of the Frobenius-Perron theorem for Metzler matrices. {\it Probability, 
statistics, and mathematics}, 17--26, Academic Press, Boston, MA, 1989. 
\bibitem{bp} \sc  M.~Barczy and G.~Pap: \rm Asymptotic behavior of critical, irreducible multi-type continuous state and 
continuous time branching processes with immigration. {\it Stoch. Dyn.}16, no. 4, (2016). 
\bibitem{be} \sc J.~Bertoin: \it L\'evy processes. \rm Cambridge University Press, Cambridge, 1996. 
\bibitem{cpu} \textsc{M.E. Caballero, J.L. P\'erez Garmendia, and G.Uribe Bravo}:  Affine processes on 
$\mathbb{R}_+^m\times\mathbb{R}^n$ and multiparameter time changes. {\it Ann. Inst. Henri Poincar\'e Probab. Stat.} 53 
no. 3, 1280--1304, (2017). 
\bibitem{ch} \textsc{L.~Chaumont}: \textsl{Breadth first search coding of multitype forests with application to Lamperti 
representation}.  In memoriam Marc Yor--S\'eminaire de Probabilit\'es XLVII, 561--584, Lecture Notes in Math., 2137, Springer, 
Cham, 2015. 
\bibitem{cam2} \textsc{L.~Chaumont and M.~Marolleau}: Extinction times of multitype, continuous-state branching processes. 
{\it In preparation}. 
\bibitem{cl} \textsc{L.~Chaumont and R.~Liu}: \textsl{Coding multitype forests: application to the law of the total population
of branching forests}. To appear in Transactions of the American Mathematical Society, (2015). 
\bibitem{gt} \sc N.~Gabrielli and J.~ Teichmann: \rm Pathwise construction of affine processes. \it Preprint \rm 
arXiv:1412.7837, (2014).
\bibitem{kx} \sc  D.~Khoshnevisan and Y.~Xiao: \rm Level sets of additive L\'evy processes. {\it Ann. Probab.} 30, no. 1, 
62--100, (2002). 
\bibitem{ku} \sc  T.G.~Kurtz: \rm The optional sampling theorem for martingales indexed by directed sets. {\it Ann. Probab.} 
8, no. 4, 675--681, (1980). 
\bibitem{sa} \sc K.I.~Sato: \rm L\'evy processes and infinitely divisible distributions. 
Cambridge Studies in Advanced Mathematics, 68. {\it Cambridge University Press, Cambridge}, 2013.
\bibitem{sk} \sc  A.V.~Skorohod: \rm Limit theorems for stochastic processes with independent increments. 
{\it Teor. Veroyatnost. i Primenen.} 2, 1957,145--177.
\end{thebibliography}
\end{document}